\documentclass[12pt]{article}

\usepackage[margin=1.20in]{geometry}

\usepackage{amssymb,amsmath}
\usepackage{xcolor}

\usepackage{amsthm,enumitem}
\usepackage{hyperref}
\usepackage{mathrsfs}
\usepackage{textcomp}
\hypersetup{
    colorlinks,%
    citecolor=black,%
    filecolor=black,%
    linkcolor=black,%
    urlcolor=black
}
\usepackage[abbrev,nobysame]{amsrefs}

\usepackage{verbatim}

\usepackage{float}
\usepackage{tikz}
\usetikzlibrary{calc}

\usepackage{sectsty}

\usepackage{scalerel}[2014/03/10]
\usepackage[usestackEOL]{stackengine}
\def\intavg{\,\ThisStyle{\ensurestackMath{%
    \stackinset{c}{0\LMpt}{c}{0\LMpt}{\SavedStyle-}{\SavedStyle\phantom{\int}}}%
    \setbox0=\hbox{$\SavedStyle\int\,$}\kern-\wd0}\int}

\makeatletter

\newdimen\bibspace
\makeatother

\makeatletter

\newtheorem{Theorem}{Theorem}[section]
\newtheorem{Lemma}[Theorem]{Lemma}
\newtheorem{Proposition}[Theorem]{Proposition}
\newtheorem*{Assumption*}{Assumption (H)}
\newtheorem{Definition}[Theorem]{Definition}
\newtheorem{ex}[Theorem]{Example}

\newtheorem{Remark}[Theorem]{Remark}

\newtheorem{conjecture}[Theorem]{Conjecture}

\def\XXint#1#2#3{{\setbox0=\hbox{$#1{#2#3}{\int}$}
  \vcenter{\hbox{$#2#3$}}\kern-.5\wd0}}

           \newcommand{\ud}{\mathrm{d}}
\newcommand{\be}{\begin{equation}}      \newcommand{\ee}{\end{equation}}

\newcommand{\T}{\mathcal{T}}
\newcommand{\R}{\mathbb{R}}              
\newcommand{\C}{\mathcal{C}}

\newcommand{\E}{\mathcal{E}}
\newcommand{\G}{\mathcal{G}}

\begin{document}

\title{\textbf{A Liouville theorem for convex functions with periodic Monge-Amp\`ere measure}\bigskip}

\author{\medskip  Tianling Jin\footnote{T. Jin was partially supported by NSFC grant 12122120, and Hong Kong RGC grants GRF 16304125, GRF 16306320 and GRF 16303822.}, \quad YanYan Li\footnote{Y.Y. Li is partially supported by NSF grant DMS-2247410.}, \quad
Hung V. Tran\footnote{H.V. Tran is partially supported by NSF grant DMS-2348305.}, \quad Xushan Tu}

\date{\today}

\maketitle

\begin{abstract}

We study global convex solutions of the Monge-Amp\`ere equation 
\[ 
\det D^2 u = \mu \quad \text{in } \mathbb{R}^n,
\]
where $\mu \not\equiv 0$ is a nonnegative locally finite periodic Borel measure on $\mathbb{R}^n$. We prove a Liouville-type theorem showing that every such solution admits a unique decomposition, up to an additive constant, as the sum of a quadratic polynomial and a periodic function. This extends earlier results of Caffarelli-Li and Li-Lu, which required $\mu$ to have a density with regular or bounded logarithm, to the full generality of periodic measures, allowing degeneracy and singularities. A key ingredient is a new dichotomous Harnack-type inequality for linearized Monge-Amp\`ere equations with nonnegative periodic measures, which compensates for the failure of doubling and engulfing properties in the degenerate setting.

In the extremal example where $\mu$ is the periodic Dirac measure supported on the integer lattice, we show that the solutions, up to addition of a linear function, are in one-to-one correspondence with Dirichlet-Voronoi tilings of $\mathbb{R}^n$.


\medskip


\end{abstract}

\section{Introduction}
A recurring phenomenon in geometry and analysis is that global constraints combined with ellipticity force rigidity at large scales. An early instance of this principle in geometry appears in the classical works of Morse \cite{Morse1924} and Hedlund \cite{hedlund1932} on globally minimizing geodesics on surfaces. Specifically, in two-dimensional periodic media, every globally minimizing geodesic is line-like: it stays within a uniformly bounded distance of a Euclidean straight line, and hence admits a well-defined asymptotic direction. In a geometric variational setting, Caffarelli-de la Llave \cite{Caffarelli2001} proved the existence of plane-like minimal surfaces in periodic media, which yields, in particular, an alternative proof of the Morse-Hedlund existence of line-like globally minimizing geodesics, and extends it to higher dimensions. On the analytic side, Liouville theorems have been extended to uniformly elliptic equations in periodic media. Avellaneda-Lin \cite{avellaneda1989} established Liouville-type theorems for linear second-order uniformly elliptic equations with periodic coefficients in divergence form, showing that entire solutions of polynomial growth are polynomials with periodic coefficients and constant leading coefficients. Moser-Struwe \cite{moser1992} proved a Liouville-type theorem for a class of nonlinear uniformly elliptic equations with periodic structure, showing that any entire solution with asymptotically linear growth is affine up to a periodic perturbation. Li-Wang \cite{liwang2001} obtained corresponding classification results for linear uniformly elliptic equations with periodic coefficients in non-divergence form, giving an explicit description of all entire solutions of polynomial growth. These works underscore a common principle: global hypotheses, such as boundedness or controlled growth, and in geometric settings, minimality or convexity, often rigidly determine the structure of entire solutions to elliptic problems. 

A nonlinear counterpart of the classical Liouville theorem arises in the context of the Monge-Amp\`ere equation. The celebrated J\"orgens-Calabi-Pogorelov theorem \cite{jorgens1954losungen, calabi1958improper, pogorelov1972improper} states that every entire convex solution of
\[
\det D^2 u = 1 \quad \text{in } \mathbb{R}^n
\]
must be a quadratic polynomial. Motivated by the preceding rigidity phenomena in periodic environments, we study a Liouville-type classification for global convex solutions of the Monge-Amp\`ere equation with a periodic right-hand side. Specifically, we consider 
\begin{equation}\label{eq:periodic equation ma}
\det D^2 u = \mu \quad \text{in } \R^n,
\end{equation}
where $\mu\not\equiv 0$ is a nonnegative locally finite  Borel measure on $\R^n$ and is periodic in $n$ linearly independent directions. By exploiting the affine invariance of \eqref{eq:periodic equation ma}, we may, without loss of generality, restrict our analysis to the case where $\mu $ is periodic with respect to the integer lattice $\mathbb{Z}^n := \{(k_1,\dots,k_n): k_i \in \mathbb{Z}\} \subset \R^n$, i.e., for any Borel set $E \subset \R^n$,  
\begin{equation}\label{eq:periodic f}
\mu (E+z)=\mu (E),\quad \forall \ z \in \mathbb{Z}^n.
\end{equation}
Our goal is to classify all global convex Alexandrov solutions of \eqref{eq:periodic equation ma} under the sole assumption \eqref{eq:periodic f}, allowing $\mu$ to be degenerate and/or singular (including measures with Dirac masses, or densities that vanish on sets of positive Lebesgue measure).

When $\mu$ has a uniformly positive, bounded, periodic density, i.e.,
\begin{equation}\label{eq:boundedlog}
\mu=f(x)\,\ud x,\quad 0<\inf_{\R^n} f\le \sup_{\R^n} f<+\infty,
\end{equation}
Caffarelli-Li \cite{caffarelli2004liouville} proved that, assuming in addition that $f$ is H\"older continuous, every global convex solution of \eqref{eq:periodic equation ma} must be asymptotically quadratic and, more precisely, decomposes uniquely, up to an additive constant, as a quadratic polynomial plus a periodic function. They also conjectured that the H\"older regularity is unnecessary. This conjecture was resolved by Li-Lu \cite{li2019monge}, who established the same classification under the sole assumption \eqref{eq:boundedlog} (with $f$ periodic), without assuming H\"older continuity of $f$. Li-Lu \cite{li2019monge} further raised the question of whether such a Liouville theorem remains true in the degenerate case $\inf_{\R^n} f=0$, and more broadly for periodic right-hand sides beyond bounded densities. The present paper answers this question in the strongest possible form: we prove the Liouville classification for arbitrary nontrivial nonnegative locally finite periodic Borel measures $\mu$. 

\subsection{Main results}

Let $\mathcal{S}_+^{n\times n}$ denote the set of positive definite symmetric $n\times n$ matrices.
\begin{Definition}[Periodic solution]\label{def:periodic solution}
A \emph{periodic solution} to \eqref{eq:periodic equation ma} is a convex function $u \in C(\mathbb{R}^n)$ satisfying \eqref{eq:periodic equation ma} in the Alexandrov sense, and can be decomposed as
\begin{equation}\label{eq:periodic solution}
u(x) = v(x) +P(x),
\end{equation}
where $v$ is periodic  with respect to  $\mathbb{Z}^n$ and  $P(x)=\frac{1}{2}x^{\top}Ax+b\cdot x+c$ is a quadratic function with $A\in \mathcal{S}_+^{n\times n}$, $b\in\R^n$ and $c\in\R$.  
\end{Definition}

Note that if $u$ is a periodic solution to \eqref{eq:periodic equation ma} and is expressed as \eqref{eq:periodic solution}, then the quadratic part $P(x)$ always satisfies the compatibility condition
\begin{equation}\label{eq:A compatible}
 \det A =\mu (\mathbb{T}^n),
\end{equation} 
where $\mathbb{T}^n=\mathbb{R}^n/\mathbb{Z}^n$.

We begin with the existence and uniqueness of periodic solutions of \eqref{eq:periodic equation ma}.

\begin{Theorem}\label{thm:uniqueness within periodic}
Let $\mu \not\equiv 0$ be a nonnegative locally finite Borel measure satisfying \eqref{eq:periodic f}, and let $P(x)=\frac{1}{2}x^{\top}Ax+b\cdot x+c$ be a quadratic function with $A\in \mathcal{S}_+^{n\times n}$ satisfying \eqref{eq:A compatible}. Then there exists a periodic solution $u=v+P$ of \eqref{eq:periodic equation ma} in the sense of Definition \ref{def:periodic solution}, and $v$ is unique up to an additive constant.
\end{Theorem}

Existence of periodic solutions follows from the proof of Theorem 2.1 in Li \cite{li1990existence} and an approximation procedure which will be given at the beginning of Section \ref{sec:main theorem}. Uniqueness was previously known in the uniformly positive bounded periodic density regime treated in \cite{li1990existence} and \cite{li2019monge}. The novelty here is that no absolute continuity or nondegeneracy is assumed: $\mu$ may have singular and degenerate parts.

The main result of the paper is the following Liouville-type classification of global solutions of \eqref{eq:periodic equation ma}.

\begin{Theorem}\label{thm:solutions are periodic}
Let $\mu \not\equiv 0$ be a nonnegative locally finite Borel measure satisfying \eqref{eq:periodic f}, and let $u$ be a global convex solution of \eqref{eq:periodic equation ma}. Then $u$ is a periodic solution in the sense of Definition \ref{def:periodic solution}.
\end{Theorem}

This extends the Liouville theorems of Caffarelli-Li \cite{caffarelli2004liouville} and Li-Lu \cite{li2019monge} from uniformly positive bounded periodic densities to general periodic measures, thereby allowing both degeneracy and singularities. In particular, it provides a complete affirmative answer to the question posed in \cite{li2019monge} in the degenerate periodic setting.

\begin{Remark} 
When $\mu \equiv 0$, equation \eqref{eq:periodic equation ma} reduces to the homogeneous Monge-Amp\`ere equation $\det D^2u = 0$. The solutions are precisely those functions that are linear along certain directions (see, e.g., Caffarelli-Nirenberg-Spruck \cite{caffarelli1986dirichlet}). 
In particular, writing $x = (x',x_n) \in \R^{n-1} \times \R$, any convex function of the form $u(x) = w(x')$ is a solution, and such solutions are not necessarily periodic.
\end{Remark}

\begin{ex}\label{ex:periodicdelta}
An extremal example is when $\mu$ is the $\mathbb{Z}^n$-periodic Dirac measure supported on the integer lattice. We discuss this case in detail in Section \ref{sec:periodicdelta}. There, we describe these solutions explicitly in terms of Dirichlet-Voronoi tilings and their dual Delaunay decompositions. Conversely, every Dirichlet-Voronoi tiling of $\R^n$ is affinely equivalent to the subdifferential tiling of such a solution. Precise formulas are given in Theorems \ref{thm:periodicdelta-tiling} and \ref{thm:periodicdelta-formula}.\end{ex}

\subsection{Key difficulties and proof ideas}

Since $\mu$ is assumed to be nonnegative,  equation \eqref{eq:periodic equation ma} represents a possibly degenerate Monge-Amp\`ere equation, for which many arguments in  \cite{caffarelli2004liouville,li2019monge} do not apply.  The principal challenge stems from the behavior of $u$ on small sections: the associated measure $\mu$ may not satisfy the doubling condition. This failure implies that sections lose the engulfing property, consequently invalidating the Vitali or  Besicovitch covering lemma which is essential for the Calder\'on-Zygmund decomposition on sections. 
Therefore, we cannot employ directly the weak Harnack inequalities of Caffarelli-Guti\'errez \cite{caffarelli1997properties} for subsolutions or supersolutions of linearized Monge-Amp\`ere equations.  

Our approach is to exploit periodicity to recover a substitute theory on large sections, where periodic averaging becomes effective. We derive truncated weak Harnack inequalities for both supersolutions and subsolutions on large sections via a maximal-type operator (see Definition \ref{def:phi homogenization}). This constitutes one of the key innovations of the present work. For supersolutions, we address the challenge by excluding contributions from small sections. Theorem \ref{thm:harnack super} establishes decay estimates for homogenized level-sets of supersolutions within our framework. For subsolutions, for which standard $L^{\infty}$ bounds fail, Theorem \ref{thm:harnack sub} establishes an alternative dichotomy regarding their growth behavior: either the $L^{\infty}$ bound holds or the $L^{\infty}$ norm grows exponentially. Finally, we obtain  in Theorem \ref{thm:harnack} a Harnack inequality with a quantitatively controlled error term.

After overcoming this critical difficulty arising from the degeneracy of equation \eqref{eq:periodic equation ma}, the rest proof focuses on analyzing the second-order difference quotient $\Delta^2_e u$ and the deviation between $u$ and its periodic counterpart. As observed in \cite{caffarelli2004liouville}, these quantities naturally arise as subsolutions and supersolutions to carefully selected linearized Monge-Amp\`ere equations. This is another point where the periodicity of the measure plays a crucial role. We then establish the quadratic behavior of solutions near infinity, extending the arguments in \cite{caffarelli2004liouville,li2019monge} for positive measures $\mu=f(x)\ud x$ to  the degenerate case considered here.  In this step, we must also overcome additional difficulties due to the degeneracy of
 \eqref{eq:periodic equation ma}. Consequently, when Theorem \ref{thm:harnack} is applied to the difference of two solutions, exponential growth is precluded. Hence, the classical Harnack inequality holds for this difference, and our theorems would follow. 

This paper is organized as follows. In Section \ref{sec:harnack}, we study convex functions whose Monge-Amp\`ere measure has small periods, establishing a dichotomous Harnack-type  inequality for linearized Monge-Amp\`ere equations. Section \ref{sec:semiconvavity} investigates  the uniqueness of the compatible quadratic part of each solution in the sense of Proposition \ref{prop:semiconcave}. In Section \ref{sec:main theorem}, we prove  Theorems \ref{thm:uniqueness within periodic} and \ref{thm:solutions are periodic}. Finally, in Section \ref{sec:periodicdelta}, we treat the extremal case in which $\mu$ is the periodic Dirac measure supported on the integer lattice and show how it connects to tilings of Euclidean space.

\section{Harnack inequalities for homogenized functions}\label{sec:harnack}

In this section, we will investigate convex functions $\phi$ whose Monge-Amp\`ere measures 
\begin{align}\label{eq:MAmeasure}
\mu_\phi(E):=|\partial\phi(E)|\qquad\text{for Borel }E
\end{align}
are periodic, with the corresponding lattice satisfying a small-period assumption.
For such functions, 
we will establish a dichotomy for the weak Harnack principle of solutions to the linearized Monge-Amp\`ere equation  
\begin{equation}\label{eq:linearized ma phi}
L_{\phi} v:= \Phi^{ij}v_{ij}=0,  
\end{equation}
where $\Phi^{ij}$ denotes the entries of the cofactor matrix of $D^2\phi$.

\subsection{Small-period structure}

Let $\varepsilon_1, \cdots, \varepsilon_n \in \R^n$ be linearly independent vectors, and $\Gamma$ be the lattice 
\[
\Gamma :=\Gamma_{\{\varepsilon_1, \cdots, \varepsilon_n\}}=\left\{\sum_{i=1}^n k_i \varepsilon_i:\; k_1,\cdots,k_n \in \mathbb{Z}\right\}. 
\]
The \emph{fundamental domain} with respect to the basis $\{\varepsilon_1, \cdots, \varepsilon_n\}$ is denoted by
\[
D(\Gamma): = \left\{x \in \R^n :\;   x = \sum_{i=1}^n t_i \varepsilon_i,\ 0 \leq t_i < 1 \right\}.
\]
In the literature, $D(\Gamma)$ is also referred to as the unit cell.
For a convex set $S \subset \R^n$, we define the \emph{relative size} of $S$ as
\begin{equation}\label{eq:lattice-relative size}
\kappa(S):=\kappa_{\Gamma}(S) =\sup \left\{t > 0 :\; \exists\ y \in \R^n \text{ such that } tD(\Gamma) + y \subset S \right\},
\end{equation}
where $tD(\Gamma) = \{ tx :\; x \in D(\Gamma) \}$.  
The relative size $\kappa(S)$ is affine invariant.

Let $\mu \not\equiv 0$ be a nonnegative locally finite Borel measure. We say that $\mu$ is $\Gamma$-periodic (or simply, periodic) if
\begin{equation}\label{eq:period gamma}
\mu (E + \varepsilon_i) =\mu (E) 
\end{equation}
holds for every $ i=1,\cdots,n$ and every Borel set $E \subset \R^n$. 

Let $\Omega\subset\R^n$ be a bounded open convex set 
and $\phi$ is the convex solution of
\begin{equation}\label{eq:potentialfunction}
\det D^2\phi=\mu\quad\mbox{in }\Omega,\quad
 \phi=0 \quad\mbox{on }\partial\Omega.
\end{equation}
By Alexandrov’s maximum principle, for $x\in \Omega$,
\begin{equation}\label{Alexandrov}
 \phi(x) \geq   -C(n)\operatorname{diam}(\Omega)^{\frac{n-1}{n}}\operatorname{dist}(x,\partial \Omega)^{\frac{1}{n}} \mu(\Omega)^{\frac{1}{n}}.
\end{equation}
By the convexity of $\phi$, using $ \phi=0 $ on $\partial\Omega$, we have $\phi\le 0$ in $\Omega$. It follows that $\phi$ is continuous in $\overline\Omega$, and is locally uniformly Lipschitz continuous in $\Omega$.

\begin{Definition}\label{def:ELambda0}
Let $\Lambda > 0$. We define $\mathcal{P}_\Lambda$ to be the set of all quadruples
\[
(\phi, \Omega,  \mu,  \Gamma)
\]
such that:
\begin{itemize}
    \item $\Gamma$ is a lattice generated by $n$ linearly independent vectors $\{\varepsilon_1, \cdots, \varepsilon_n\}\subset\R^n$;
        \item $\mu$ is a nonzero nonnegative locally finite $\Gamma$-periodic Borel measure on $\mathbb{R}^n$;       
    \item $\Omega \subset \mathbb{R}^n$ is a bounded open convex set satisfying $\kappa(\Omega)\ge\Lambda^2$;
    \item $\phi \in C(\overline\Omega)$ is the convex solution of \eqref{eq:potentialfunction}.
\end{itemize}
\end{Definition}

Throughout this section, we maintain the following \emph{normalization assumptions}:
\begin{equation}\label{eq:normalized assumption}
\mu (D(\Gamma))   = |D(\Gamma)|,
\end{equation}
where $|D(\Gamma)|$ stands for the Lebesgue measure of $D(\Gamma)$, and
\begin{equation}\label{eq:normalized assumptionO}
B_1 \subset \Omega \subset B_n.
\end{equation}
As a consequence of \eqref{eq:normalized assumptionO}, we have
\[
\kappa (B_1)\le \kappa(\Omega)\le n\kappa (B_1).
\]
For convenience, we define
\begin{Definition}\label{def:ELambda}
Let $\Lambda > 0$. We define
\[
\E_\Lambda:=\{(\phi, \Omega,  \mu,  \Gamma)\in \mathcal{P}_\Lambda: \mu\mbox{ satisfies }\eqref{eq:normalized assumption}, \mbox{ and }\Omega \mbox{ satisfies } \eqref{eq:normalized assumptionO}\}.
\]
\end{Definition}

\noindent When we write $\phi\in \mathcal{P}_\Lambda$ or $\phi\in \E_\Lambda$, we implicitly fix such a quadruple.

\begin{Theorem}
\label{thm:homogenization-periodic-measures}
There exist constants $\Lambda_*,\beta,C>0$, depending only on $n$, such that
the following holds. Let $(\phi,\Omega,\mu,\Gamma)\in\E_\Lambda$ with
$\Lambda\ge\Lambda_*$. Let $\bar\phi$ be the convex solution of
\[
\det D^2\bar\phi=1
\quad \text{in }\Omega,
\qquad
\bar\phi=0
\quad \text{on }\partial\Omega.
\]
Then
\[
\|\phi-\bar\phi\|_{L^\infty(\Omega)}
\le
C\kappa_\Gamma(\Omega)^{-\beta}.
\]
\end{Theorem}

For measures with bounded positive periodic densities, this is exactly Theorem 0.3 in \cite{caffarelli2004liouville}, 
written in terms of the relative size $\kappa_\Gamma(\Omega)$. One can check that the assumptions on the densities in the proof of Theorem 0.3 in \cite{caffarelli2004liouville} is used only for
the boundary barrier estimate, which in the present periodic measure setting is replaced directly by the  Alexandrov estimate \eqref{Alexandrov}.
The periodic corrector and the local comparison argument are unchanged. 
The passage from smooth positive densities to general periodic Borel measures follows from standard approximating arguments. The proof of Theorem \ref{thm:homogenization-periodic-measures} is sketched in Appendix \ref{app:quadratic-asymptotics}.

Throughout this section, we fix a constant \(\Lambda>1\), chosen sufficiently large depending only on the dimension \(n\). Under the normalization assumption \eqref{eq:normalized assumptionO}, the condition that $\kappa(\Omega)\ge\Lambda^2$ for \(\Lambda\) sufficiently large is equivalent to a smallness assumption on the periods of the lattice \(\Gamma\), namely that
\[
\|\{\varepsilon_1,\dots,\varepsilon_n\}\|
:= \|\varepsilon_1\|+\cdots+\|\varepsilon_n\|
\]
be sufficiently small. We refer to this as the \emph{small-period assumption}.

For a convex function $\phi\in C(\overline\Omega)$ and $y\in\Omega$,  we fix a subgradient \(p \in \partial \phi(y)\) and define the \emph{section} of $\phi$ at $y$ of height $h>0$ with subgradient $p$ to be the convex set
\[
S_h(y):=S_{h,p}^{\phi}(y)=\{x \in \overline{\Omega}:\; \phi(x) < \phi (y)+p\cdot (x-y)+h\}.
\]
To ensure a Borel measurable choice, for each $y\in\Omega$ we take $p=p(y)$ to be the centroid of $\partial\phi(y)$ with respect to the Lebesgue measure on its affine hull. This selection is compatible with affine changes of variables and addition or subtraction of affine functions. We omit $p$ from the subscript in the rest of the paper. Throughout this work, we exclusively consider \emph{large} sections $S_h(y)$ that are contained in $\Omega$, meaning their height $h$ satisfies  
\[
\check{h}(y) < h \leq  \hat{h}(y).
\]  
Here, $\hat{h}(y)$ denotes the maximal admissible height for the interior sections associated with $y$, defined by
\[
\hat{h}(y) := \sup \bigl\{ h > 0 : S_h(y) \subset\Omega \bigr\};
\]
and  $\check{h}(y)$ denotes the minimal  height for the sections being large, defined by
\begin{equation}\label{eq:large height}
\check{h}(y) := \inf \bigl\{ h \in (0,\hat{h}(y)]: \kappa(S_h(y)) \geq \Lambda \bigr\}
\end{equation}
whenever $\bigl\{ h \in (0,\hat{h}(y)]: \kappa(S_h(y)) \geq \Lambda \bigr\}$ is not empty. Otherwise, we define $\check{h}(y)=+\infty$.

\begin{Proposition}\label{prop:basic}
There exist positive constants $\Lambda_0$ and $\gamma$, depending only on
$n$, such that the following holds. For every open set
$\Omega'\Subset\Omega$, there exists a constant $\Lambda_{\Omega'}\ge \Lambda_0$,
depending only on $n$ and
$\operatorname{dist}(\Omega',\partial\Omega)$, such that, whenever
$\phi\in\E_\Lambda$ and $\Lambda\ge\Lambda_{\Omega'}$, there exist
positive constants $c$ and $C$, depending only on $n$ and
$\operatorname{dist}(\Omega',\partial\Omega)$, for which the following
properties hold for every $y\in\Omega'$:
\begin{enumerate}
\item[(i).] Uniform height bounds:   
\[
c\le \hat h(y)\le C .
\]
\item[(ii).] Polynomial growth: for
$\hat h(y)\geq h\geq C\Lambda_0^2\kappa(\Omega)^{-2}$, we have
$\kappa(S_h(y))\geq\Lambda_0$ and
\[
B_{ch^{1/2}}(y)
\subset S_h(y)
\subset B_{Ch^{1/2}}(y).
\]
Moreover, $c\kappa(\Omega) h^{\frac{1}{2}}\le \kappa( S_h(y))\le C\kappa(\Omega) h^\frac{1}{2} $ and $c\Lambda^2\kappa(\Omega)^{-2}\le \check h(y)\le C\Lambda^2\kappa(\Omega)^{-2}$.
\item[(iii).] Uniform density estimates:   if
$\kappa (S_h(y))\ge \Lambda_0$ and $h \leq \hat{h}(y)$,
\[
c(n) h^{\frac{n}{2}} \leq |S_h(y)|\leq C(n) h^{\frac{n}{2}} ,\quad c(n)  h^{\frac{n}{2}} \leq \mu(S_h(y))\leq C(n)  h^{\frac{n}{2}}.
\] 
\item[(iv).] Engulfing property: suppose
$\kappa (S_h(y))\ge \Lambda_0$ and $h \leq \hat{h}(y)$. If $S_{c(n)h}(z) \cap S_{c(n)h}(y) \neq \emptyset$, then
\[
S_{c(n)h}(z) \subset S_{h}(y).
\]
\item[(v).] Quantitative quadratic approximation: for every $\varepsilon>C\kappa(\Omega)^{-\frac{\beta}{2}}$, there exist
a convex quadratic polynomial $P_{y,\varepsilon}$ with $\det D^2P_{y,\varepsilon}=1$ and
$D^2P_{y,\varepsilon}\leq C I$, and a constant $C_\varepsilon>0$, depending only on
$n$, $\varepsilon$, and $\operatorname{dist}(\Omega',\partial\Omega)$, such that
\[
|\phi(x)-P_{y,\varepsilon}(x)|
\le
\varepsilon |x-y|^2+C |x-y|^{2+\gamma}+C_\varepsilon \kappa(\Omega)^{-2}
\quad \text{for all } x\in\Omega .
\]
Here, $\beta$ is the one in Theorem \ref{thm:homogenization-periodic-measures}.
\end{enumerate}
\end{Proposition}

\begin{proof}
Fix $y\in\Omega'$. Applying Theorem~\ref{thm:homogenization-periodic-measures}
and using the local smoothness of $\bar\phi$, we obtain
\[
c\le \hat h(y)\le C\quad\mbox{and}\quad  B_c(y)\subset S_{\hat{h}(y)}(y)\subset B_C(y),
\]
where $c,C>0$ depend only on $n$ and $\operatorname{dist}(\Omega', \partial\Omega)$. 
This proves (i). Moreover, choose $\Lambda_{\Omega'}$ sufficient large so that 
\[
\Lambda_*r^{-\frac{2}{\beta}}\le\Lambda_0\le c\kappa(\Omega)\le \kappa(S_{\hat{h}(y)}(y))\le C\kappa(\Omega),
\]
where $\Lambda_*$ and $\beta$ are those in Theorem~\ref{thm:homogenization-periodic-measures}, and $r>0$ is a sufficiently small constant depending only on $n$ and will be chosen later. 

After translating the coordinates, subtracting a suitable affine function, and
normalizing the height scale, we may assume, without loss of generality, that
\[
y=0,\qquad \phi(0)=0,\qquad
\phi\ge 0,\qquad \hat h(0)=1.
\]
For $0<h<1(=\hat h(0))$, we set
\[
S_h:=\{x\in\Omega:\phi(x)<h\},\qquad \kappa_h=\kappa(S_h).
\]
We will work for those $h$ satisfying $\kappa_h\ge\Lambda_0$. 

Let $\bar\phi_h$ be the convex solution of
\[
\det D^2\bar{\phi}_h=1
\quad \text{in }S_h,
\qquad
\bar{\phi}_h=h
\quad \text{on }\partial S_h .
\]
Note that $S_h=S_{\bar{h}}^{\bar{\phi}_h}(\bar y)$ is also a section of $\bar{\phi}_h$ at some point $\bar y$ of height $\bar h$. The classical estimate for $\bar{\phi}_h$ implies
that its height satisfies $c(n) |S_h|^{2/n}\leq \bar h \leq C(n) |S_h|^{2/n}$. After normalizing by a linear transformation $T_h$ that
\[
B_1(0)\subset \Omega_h:=T_h^{-1} S_h\subset B_n(0),\qquad c(n) \bar{h}^{\frac{n}{2}}\le \det  DT_h \le C(n) \bar{h}^{\frac{n}{2}},
\]
and applying Theorem~\ref{thm:homogenization-periodic-measures} to $(\det DT_h)^{-\frac{2}{n}}\phi(T_hx)$ and $(\det DT_h)^{-\frac{2}{n}}\bar{\phi}_h(T_hx)$, we obtain
\[
|h-\bar{h}| \leq \|\phi-\bar{\phi}_h\|_{L^\infty(S_h)}
\le
C(n)\bar{h}\kappa_h^{-\beta},
\]
which implies
\begin{equation}\label{eq:approximated estimates}
\|\phi-\bar{\phi}_h\|_{L^\infty(S_h)}
\le
C(n)h\kappa_h^{-\beta}, \qquad |h-\bar{h}| \leq C(n)h\kappa_h^{-\beta}.
\end{equation} 
Consequently, $c(n) |S_h|^{2/n}\leq  h \leq C(n) |S_h|^{2/n}$. If $\Lambda_0$ is sufficiently large depending only on $n$, then $c(n)|S_h|\le \mu(S_h) \le C(n)|S_h|.$ This proves (iii). 

Furthermore, Alexandrov's estimate
\eqref{Alexandrov}  gives
\begin{equation}\label{eq:distlow}
\mbox{dist}(T_h^{-1}(0),\partial \Omega_h)\ge c(n).
\end{equation}
Then the $C^2$ regularity of $(\det DT_h)^{-\frac{2}{n}}\bar{\phi}_h(T_hx)$ near $T_h^{-1}(0)$ implies that, for the fixed small $r>0$,
\[
B_{c(n)r^{1/2}}(T_h^{-1}(0))\subset \{ (\det T_h)^{-\frac{2}{n}} \bar{\phi}_h(T_hx)  <r  \}\subset B_{C(n)r^{1/2}}(T_h^{-1}(0)).
\]
Since $\kappa_h^{-\beta}\le \Lambda_*^{-\beta}r^2$, we have by using Theorem \ref{thm:homogenization-periodic-measures} that
\[
B_{c(n)r^{1/2}}(T_h^{-1}(0))\subset \{ (\det T_h)^{-\frac{2}{n}}  \phi(T_hx)  <r  \}\subset B_{C(n)r^{1/2}}(T_h^{-1}(0)).
\] 
Consequently, we obtain the scaling-invariant estimates
\begin{equation}\label{eq:induction section}
c(n)r^{1/2} S_{h}\subset S_{rh} \subset C(n)r^{1/2} S_{h}, \qquad c(n)r^{\frac{1}{2}}\kappa(S_{h})  \leq \kappa(S_{rh}) \leq C(n)r^{\frac{1}{2}}\kappa(S_{h}).
\end{equation}

The same comparison with $\bar{\phi}_h$, together with the engulfing property for the rescaled function $(\det T_h )^{-\frac{2}{n}}\bar{\phi}_h(T_hx)$, gives the engulfing property for $(\det T_h )^{-\frac{2}{n}}\phi(T_hx)$. Consequently, part (iv) holds.

We are going to establish parts (ii) and (v) below. 

Since $c(n)\le |S_{1}|\le C(n)$ and recalling \eqref{eq:distlow}, after a unimodular  affine transformation, we may, without loss of generality, assume that 
\[
B_{c(n)}(0)\subset S_{1}\subset B_{C(n)}(0).
\]
After this normalization, the constants $c$ and $C$ below will depend only on $n$, even though they may change from line to line.

We define
\[
E_h(w):=
\inf_{P\in\mathcal P}
\frac{1}{h}\|w-P\|_{L^\infty(S_{h})},
\]
where
\[
\mathcal P
:=
\left\{
P(x)=\frac12 x^\top A_Px+\ell_P(x)
:\ A_P\in\mathcal S_+^{n\times n},\ \det A_P=1,\ 
\ell_P \ \text{is affine}
\right\}.
\]
By \eqref{eq:approximated estimates}, we have
\[
|E_h(\phi)-E_h(\bar{\phi}_h)| \leq  \frac{1}{h}\|\phi-\bar{\phi}_h\|_{L^\infty(S_{h})}
\le C\kappa_h^{-\beta}.
\]

We are going to show that there exists $\theta\in(0,1)$ such that if $r$ is small
and $\kappa_h\ge \Lambda_0$, then
\begin{equation}\label{eq:induction claim}
E_{rh}(\phi)\le \theta E_h(\phi)+C\kappa_h^{-\beta}.
\end{equation}
Let us choose $P_h\in\mathcal P$ such
that
\[
\|\phi-P_h\|_{L^\infty(S_h)}\le 2hE_h(\phi).
\]
 Denote
\[
\widetilde\phi_h(x):=(\det  T_h)^{-\frac{2}{n}}\bar\phi_h(T_hx),
\qquad
\widetilde P_h(x):=(\det T_h)^{-\frac{2}{n}}P_h(T_hx).
\]
Then,
\[
\|\widetilde\phi_h-\widetilde P_h\|_{L^\infty(\Omega_h)} \le \frac{C\|\bar\phi_h-P_h\|_{L^\infty(S_h)}}{h}
\leq \frac{C\| \phi-P_h\|_{L^\infty(S_h)}}{h}+C\kappa_h^{-\beta} \leq  C(E_h(\phi)+\kappa_h^{-\beta}).
\]
Let $\widetilde Q_h$ and $\widetilde Q_h-\widetilde P_h$ denote the quadratic Taylor polynomials of $\widetilde\phi_h$ and
$\widetilde\phi_h-\widetilde P_h$ at $T_h^{-1}(0)$, respectively. Then $\widetilde Q_h \in \mathcal P$.
Since $\det D^2 \widetilde\phi_h=\det D^2 \widetilde P_h=1$, $\widetilde\phi_h-\widetilde P_h$ satisfies a uniformly elliptic equation with H\"older continuous coefficients. The interior $C^{2,\alpha}$ Schauder estimate yields
\[
|\widetilde\phi_h-\widetilde Q_h(x)|
\le  C\left(E_h(\phi)+\kappa_h^{-\beta}\right)|x-T_h^{-1}(0)|^{2+\alpha},
\]
in $B_c(T_h^{-1}(0))$.
Note from \eqref{eq:induction section} that 
\[
B_{cr^{1/2}}(T_h^{-1}(0))\subset T_h^{-1}S_{rh}(0)\subset B_{Cr^{1/2}}(T_h^{-1}(0)).
\]
Taking $r>0$ sufficiently small, and writing
\[
P_{rh}(x):=(\det T_h)^{\frac{2}{n}}\widetilde Q_{h}(T_h^{-1}x),
\]
we then derive
\[
\|\bar\phi_h-P_{rh}\|_{L^\infty(S_{rh})}
\le  Cr^{1+\alpha/2}h \left(E_h(\phi)+\kappa_h^{-\beta}\right).
\]
After choosing $r$ so small that $$\theta:=Cr^{\alpha/2}<1,$$ and using \eqref{eq:approximated estimates} again, we conclude that
\[ 
E_{rh}(\phi)
\le (rh)^{-1}
\left(
\|\bar\phi_h-P_{rh}\|_{L^\infty(S_{rh})}
+Ch\kappa_h^{-\beta}
\right)  \leq \theta E_h(\phi)+C\kappa_h^{-\beta}. 
\]
This gives \eqref{eq:induction claim}.

Denote $h_i:=r^i$, $E_i:=E_{h_i}(\phi)$ and $\kappa_i=\kappa(S_{h_i})$. For $r$ small, iterating \eqref{eq:induction section} yields the rough
estimate
\[
c^{m-j} r^{\frac{(m-j)}{2}}\kappa(S_{h_j})
\le \kappa(S_{h_m})
\le C^{m-j} r^{\frac{m-j}{2}}\kappa(S_{h_j}),
\qquad j=0,1,\cdots, m.
\]
Hence, if $r$ is small, then
\begin{equation}\label{eq:roughgeo}
r^{\frac{2(m-j)}{3}}\kappa(S_{h_j})
\le \kappa(S_{h_m})
\le r^{\frac{m-j}{3}}\kappa(S_{h_j}),
\qquad j=0,1,\cdots, m.
\end{equation}
Iterating \eqref{eq:induction claim} along the discrete
scales $r,r^2,\dots$, we derive that
\[
E_i\le C \theta^{i-1}E_1(\phi)+ C\sum_{j=1}^{i-1} \theta^{i-1-j}\kappa_{j}^{-\beta} \leq C\left(\theta^{i-1}+\kappa_{i-1}^{-\beta}\right) \leq  C\left(r^{i\gamma}+\kappa_{i-1}^{-\beta}\right)
\]
for some $\gamma>0$, provided that $\kappa_{i}\ge\Lambda_0$.

For each $i=0,1,\cdots, m$, where $m\in\mathbb N$ is chosen so that
 $\kappa(S_{h_m})\ge \Lambda_0$, we can choose
\[
P_i(x)=\frac12 x^\top A_ix+\ell_i(x)\in\mathcal P
\]
such that
\[
\|\phi-P_i\|_{L^\infty(S_{h_i})}
\le Ch_iE_i.
\]

We first remove the affine part of $P_i$. In the normalized section
$\Omega_{h_i}=T_{h_i}^{-1}S_{h_i}$, define
\[
\widetilde P_i(y):=(\det T_{h_i})^{-\frac{2}{n}}P_i(T_{h_i}y).
\]
Then
\[
\| (\det D T_{h_i})^{-\frac{2}{n}}\phi(T_{h_i}y)-\widetilde P_i\|_{L^\infty(\Omega_{h_i})}
\le C E_i.
\]
Note that the Hessian of $\widetilde P_i$ is uniformly bounded and uniformly positive
after the normalization. Moreover, 
$y_i:=T_{h_i}^{-1}(0)$
remains a fixed positive distance away from $\partial\Omega_{h_i}$. Note that $0\in\partial\phi(0)$. Applying
Lemma~\ref{lem:uniformly converge implies gradient converge} to in a fixed ball around $y_i$ gives
\[
|D\widetilde P_i(y_i)|\le C E_i^{\frac{1}{2}}.
\]
This implies, after pulling back to the original variables, 
\[
\big\|\phi-\frac12 x^\top A_ix\big\|_{L^\infty(S_{h_i})}
\le Ch_iE_i^{\frac{1}{2}} \leq Ch_i\left(r^{\frac{i\gamma}{2}}+\kappa_{i-1}^{-\frac{\beta}{2}}\right)
\]
This implies that
\[
\left|
x^\top(A_{i+1}-A_i)x
\right|
\le
Ch_i\left(r^{\frac{i\gamma}{2}}+\kappa_{i}^{-\frac{\beta}{2}}\right)
\qquad
\text{for }x\in S_{h_{i+1}}.
\]
and
\[
\{x:x^\top A_{i}x\le c h_{i+1}\}
\subset S_{h_{i+1}}
\subset
\{x:x^\top A_{i}x\le C h_{i+1}\}.
\]
Testing the preceding quadratic estimate at points of the form
$x=\sqrt{c}h_{i+1}^{1/2}A_{i}^{-1/2}\xi$, $|\xi|=1$, and using the above inclusion, we
get
\[
\left|
\xi^\top
\left(A_{i}^{-1/2}A_{i+1}A_{i}^{-1/2}-I\right)
\xi
\right|
\le
C \left(r^{\frac{i\gamma}{2}}+\kappa_{i}^{-\frac{\beta}{2}}\right).
\]
Since the matrices are symmetric, 
\[
\left\|
A_{i}^{-1/2}A_{i+1}A_{i}^{-1/2}-I
\right\|
\le
C \left(r^{\frac{i\gamma}{2}}+\kappa_{i}^{-\frac{\beta}{2}}\right).
\]
Equivalently,
\[
\left(1-C \left(r^{\frac{i\gamma}{2}}+\kappa_{i}^{-\frac{\beta}{2}}\right)\right)A_{i}
\le
A_{i+1}
\le
\left(1+C \left(r^{\frac{i\gamma}{2}}+\kappa_{i}^{-\frac{\beta}{2}}\right)\right)A_{i}
\]
Iterating this estimate gives, for every fixed $m$ and every $i=0,1,\cdots, m$,
\[
\big\|
A_i^{-1/2}A_mA_i^{-1/2}-I
\big\|
\leq \prod_{j=i}^m\left(1+C\left(r^{\frac{j\gamma}{2}}+\kappa_{j}^{-\frac{\beta}{2}}\right)\right)-1
\leq
C\left(r^{\frac{i\gamma}{2}}+\kappa_{m}^{-\frac{\beta}{2}}\right),
\]
where we used \eqref{eq:roughgeo} in the last inequality. Then we have
\[
c^2I \leq cA_1 \le A_m\le CA_1 \leq  C^2I,
\]
if $r$ is sufficiently small. This implies for every $i=0,1,\cdots, m$,
\begin{equation}\label{eq:approxmated poly in rm}
\big\| \phi-\frac{1}{2} x^\top A_mx
\big\|_{L^\infty(S_{h_i})}
\le
Ch_i \left(r^{\frac{i\gamma}{2}}+\kappa_{m}^{-\frac{\beta}{2}}\right).
\end{equation}
Let $x\in S_{h_{i-1}}\setminus S_{h_{i}}$ with $i=1,\cdots,m$. By the ellipsoid comparison above and the uniform ellipticity of $A_i$,
we have $ch_{i}\le |x|^2\le C h_{i-1}$. Hence,
\begin{equation}\label{eq:outsideSm}
\big|\phi(x)-\frac{1}{2} x^\top A_mx \big|
\le C \left(|x|^{\gamma}+\kappa_{m}^{-\frac{\beta}{2}}\right) |x|^2
\quad \text{in } S_1\setminus S_{h_m} .
\end{equation}
Since $cI\le A_m\le CI$, we also conclude that, 
\[
B_{c h_i^{1/2}}(0)
\subset S_{h_i}
\subset B_{C h_i^{1/2}}(0).
\]
This gives for $r^{m}\le t<1$ and $S_t:=\{x\in\Omega:\phi(x)<t\}$ that
\begin{equation}\label{eq:quadraticdecay}
B_{c t^{1/2}}(0)
\subset S_{t}
\subset B_{C t^{1/2}}(0),\qquad ct^{\frac{1}{2}}\kappa(S_1) \leq \kappa(S_{t})\leq C t^{\frac{1}{2}}\kappa(S_1).
\end{equation}

Note that the choice of $m$ only requires $\kappa_m=\kappa(S_{h_m})\ge\Lambda_0$. Let $m$ be the largest integer such that
\[
h_{m}=r^{m}\ge Cr^{-1}\Lambda_0^2\kappa(S_1)^{-2}.
\]
We claim that
\[
\kappa_{m}\ge\Lambda_0.
\]
We can prove this claim by induction for $j=0,1,2,\cdots,m$.  It is true that $\kappa_0=\kappa(S_1)=\kappa(S_{h_0})\ge\Lambda_0$. Now suppose $\kappa_j\ge\Lambda_0$ for some $i\le m-1$.  Since, by convexity, $rS_{h_i}\subset S_{h_{i+1}}$, then we have  that
\[
\kappa(S_{h_{i+1}})\ge r\kappa(S_{h_i})
\ge cr^{\frac{i+2}{2}}\kappa(S_1)\ge cr^{\frac{m+1}{2}}\kappa(S_1)\ge \Lambda_0,
\]
where we used \eqref{eq:quadraticdecay} in the second inequality. This proves the claim.

This proves part (ii).

Finally, fix $\varepsilon\ge C \kappa(S_1)^{-\frac{\beta}{2}}$ and choose the largest $m=m_\varepsilon$ such that
\[
C\kappa_{m}^{-\frac{\beta}{2}}\leq C r^{-\frac{m\beta}{4}} \kappa(S_1)^{-\frac{\beta}{2}}\leq \frac{\varepsilon}{2},
\]
that is,
\[
r^m\approx C \varepsilon^{-\frac{4}{\beta}}\kappa(S_1)^{-2}.
\]
Therefore, for $C \kappa(S_1)^{-\frac{\beta}{2}}\le \varepsilon\le\Lambda_0^{-\frac{\beta}{2}} $, we can choose such $m_\varepsilon$. Applying the estimate \eqref{eq:approxmated poly in rm}, we obtain
\[
\big|\phi(x)-\frac{1}{2} x^\top A_mx \big|
\leq C_\varepsilon \kappa(S_1)^{-2}
\quad  \text{in } S_{h_m}.
\]
This, together with \eqref{eq:outsideSm}, proves the quantitative quadratic approximation in part (v).
\end{proof}

By restricting our consideration to  large interior sections, we maintain the validity of the Vitali covering lemma,

\begin{Lemma}[Vitali Covering]\label{lem:vitali covering Lemma}
There exist positive constants $c(n)<1$ and $\Lambda_1(n)>1$, both of which depend only on the dimension $n$, with the following property: 
Let $\phi\in\E_\Lambda$ with $\Lambda\ge\Lambda_1$. Let  $\mathcal{F} = \{S_{c(n) h_y}(y)\}_{y\in \G}$ be an arbitrary collection of sections of $\phi$ satisfying
\[
\kappa(S_{h_y}(y)) > \Lambda_1(n) \quad \text{and} \quad S_{h_y}(y) \subset \Omega \quad \text{for all } y \in \G.
\]
Then one can extract a finite subcollection $\mathcal{F}_0 = \{S_{c(n) h_k}(y_k)\}_{k=1}^{k_0}$ from $\mathcal{F}$ such that
\[
E:=\bigcup_{y \in \G} S_{c(n)h_y}(y) \subset \bigcup_{k=1}^{k_0} S_{h_k}(y_k),\quad \text{with }\{S_{c(n)h_k}(y_k)\}_{k=1}^{k_0} \text{ disjoint}.
\] 
\end{Lemma}
\begin{proof}
The proof follows from the standard Vitali covering argument. We begin by selecting $y_1\in \G $ with $h_{1}:=h_{y_1} >\frac{1}{2}\sup_{y \in \G}h_y$. 
Proceeding inductively, suppose $S_{h_i}(y_i)$ have been chosen for $i=1,\ldots,k$, and define  $\G_k= \{ y\in \G:\; S_{c(n)h_y}(y)  \not\subset \cup_{i=1}^kS_{h_{i}}(y_{i}) \}$.
We select $y_{k+1} \in \G_k$ satisfying
\[
h_{k+1} := h_{y_{k+1}} > \frac{1}{2}\sup_{y \in \G_k} h_y.
\]
This selection procedure yields at most a countable family $\{S_{h_k}(y_k)\}_{k=1}^\infty$.  
 (If $\G_k = \emptyset$ for some $k$, we set $S_{h_{k+1}}(y_{k+1}) = \emptyset$ for simplicity). 

We first verify the disjointness property. Suppose by contradiction that $S_{c(n)h_i}(y_i) \cap S_{c(n)h_j}(y_j) \neq \emptyset$ for some $i < j$. Since our selection process ensures $h_i > \frac{1}{2}h_j$, then $S_{2c(n)h_i}(y_i) \cap S_{2c(n)h_i}(y_j) \neq \emptyset$. Thus, for sufficiently small $c(n)$, the engulfing property between sections in Proposition \ref{prop:basic} implies $S_{c(n)h_j}(y_j) \subset S_{2c(n)h_i}(y_j) \subset  S_{h_i}(y_i)$. This contradicts to the selection $y_j\in \G_{j-1}$ such that $S_{c(n)h_j}(y_j)  \not\subset \cup_{i=1}^{j-1}S_{h_{i}}(y_{i})$.

Finally, observe the lower bound 
\[
\mu(S_{c(n)h_k}(y_k)) > \mu(D(\Gamma)) \quad \text{whenever} \quad \kappa(S_{c(n)h_k}(y_k)) \ge c(n)\kappa(S_{h_k}(y_k)) >c(n)\Lambda_1> 1.
\]
Combining this with the fact $\mu\left(\bigcup_{k=1}^\infty S_{c(n)h_k}(y_k)\right) \leq \mu(\Omega) < \infty$,
we conclude that the selection procedure must terminate after a finite number of steps, say at step $k_0$. 
Hence, $\G_{k_0} = \emptyset$, and the family $\{S_{h_k}(y_k)\}_{k=1}^{k_0}$ forms a finite covering of $E$.
\end{proof}

As an immediate consequence, we obtain the following lemma.
\begin{Lemma}\label{lem:covering Lemma}
There exist positive constants $C_1(n)$ and $\Lambda_2(n)$ with the following property:
Let $\phi\in\E_\Lambda$ with $\Lambda\ge\Lambda_2$. Let  $A \subset B \subset \Omega$ be measurable sets and  $0 < \lambda < 1$.  For each point $y \in A$, assume there is an associated section $S_{h_y}(y) \subset B$ satisfying $\kappa(S_{h_y}(y)) >\Lambda_2$  and
\[
 \mu(A \cap S_{h_y}(y)) \leq \lambda \mu (S_{h_y}(y)).
\]
Then, 
\[
\mu(A) \leq C_1(n) \lambda \mu(B).
\]
\end{Lemma}

\begin{proof}
Applying Lemma \ref{lem:vitali covering Lemma} to the covering $\mathcal{F} = \{S_{c(n) h_y}(y)\}_{y\in A}$ of $A$, we obtain a finite subcovering $\mathcal{F}_0 = \{S_{h_k}(y_k)\}_{k=1}^{k_0}$ with the property that $\{S_{c(n) h_k}(y_k)\}_{k=1}^{k_0}$ are disjoint. Then we have:
\[
\begin{split}
\mu(A) &\leq  \sum_{k=1}^{k_0} \mu ( A\cap  S_{h_k}(y_k))  \\
&\leq  \lambda\sum_{k=1}^{k_0} \mu (  S_{h_k}(y_k)) \\
& \leq C_1(n)\lambda \sum_{k=1}^{k_0} \mu (  S_{c(n)h_k}(y_k))  \\
&\leq C_1(n)\lambda\mu(B),
\end{split}
\]
where in the third inequality we used the estimate $\mu(S_h) \approx h^{\frac{n}{2}}$ for large interior sections in $\Omega$ and the fact that  $\kappa( S_{c(n)h_k}(y_k))\ge c(n)\kappa( S_{h_k}(y_k))\ge c(n)\Lambda_2>\Lambda_0^2$.
\end{proof}

The following is a quantitative strong maximum principle, which is of independent interest. For any convex set $\Omega$ and $c>0$, the dilation $c\Omega$ is always  considered with respect to its center of mass. Recall from \eqref{eq:MAmeasure} that for a convex function $u$, $\mu_u$ denotes its Monge-Amp\`ere measure. 

\begin{Lemma}\label{lem:convex-comparison}
For any $\delta_1>0$ and any $\delta_2\in (0,1)$, there exist positive constants $\Lambda_3(n,\delta_1,\delta_2)$ and 
$\tau(n,\delta_1,\delta_2)$ such that if $\phi\in\E_{\Lambda}$ with $\Lambda\ge \Lambda_3$, and $u$ is a convex 
function satisfying 
\[
\mu_u \leq \mu_\phi \quad \text{in } \Omega,\quad  u=0(=\phi) \quad \text{on } \partial \Omega, \quad\text{and } \mu_\phi(\Omega) -\mu_u (\Omega) \geq \delta_1,
\]
then we have the strict lower bound
\[  
\inf_{(1-\delta_2)\Omega} (u  - \phi) > \tau > 0. 
\] 
\end{Lemma}
\begin{proof}
By the comparison principle and Alexandrov's maximum principle, we obtain, for $x\in \Omega$,
\[
u(x) \geq \phi(x) \geq   -C(n)\operatorname{diam}(\Omega)^{\frac{n-1}{n}}\operatorname{dist}(x,\partial \Omega)^{\frac{1}{n}} \mu (\Omega)^{\frac{1}{n}},
\]
which implies the uniform continuity of $u$ up to the boundary.
Suppose the statement fails. Then there exist sequences of functions $\phi_k$ and $u_k$, periods $\Gamma_k$ with $\kappa_{\Gamma_k}(B_1)\to\infty$, and domains $\Omega_k$ satisfying \eqref{eq:normalized assumption}  such that for some points $x_k \in (1-\delta_2)\Omega_k$, we have $u_k(x_k) - \phi_k(x_k) \to 0$. By subsequential convergence, we obtain a limit domain and limit functions (still denoted by $\Omega$, $u$, $\phi$) satisfying $B_1(0) \subset\Omega =\{\phi<0 \} \subset B_n(0)$,
\[
\mu_u \leq \mu_\phi\equiv 1 \quad \text{in } \Omega,\quad\text{and} \quad  u=\phi=0 \quad \text{on } \partial \Omega.
\]
Additionally, using the fact that the periods tend to zero ($\kappa_{\Gamma_k}(B_1)\to\infty$), we also have
\[
\mu_\phi(\Omega) -\mu_u (\Omega) \geq \delta_1.
\]
Then $\phi\in C(\overline\Omega)\cap C^\infty(\Omega)$ satisfies $\det D^2 \phi=1$ in $\Omega$, and $u\in C(\overline\Omega)$ satisfies $\det D^2 u\le 1$ in $\Omega$ in the viscosity sense. Moreover, $u$ touches $\phi$ from above at some point $x_\infty \in (1-\delta_2)\Omega$. By the strong maximum principle in Theorem \ref{thm:strong maximum principle}, $u\equiv\phi$ in $\Omega$, from which we obtain a contradiction to $\mu_\phi(\Omega) -\mu_u (\Omega) \geq \delta_1>0$.
\end{proof}

\subsection{A new dichotomous Harnack-type  inequality}

We start with a definition.

\begin{Definition}[$\Gamma$-approximated subsolution/supersolution] \label{def:approximated solution}
Let $U\subset\R^n$ be an open set. Let $\phi$ be a  convex function in $U$ whose Monge-Amp\`ere measure $\mu_\phi$ is the restriction of a $\Gamma$-periodic Borel measure $\mu$. A  function $v\in C(U)$ is said to be a \textit{$\Gamma$-approximated subsolution} (resp. \textit{supersolution}) of equation \eqref{eq:linearized ma phi} in $U$, denoted as 
\[
L_\phi v\ge 0 \mbox{ in } U \ \  (\mbox{resp. } L_\phi v\le 0 \mbox{ in } U),
\]
if there exist a sequence of functions $\{\phi_j\}\subset C^2(U)$ satisfying $D^2\phi_j>0$ in $U$, and a sequence of functions $\{v_j\}\subset C^2(U)$ such that:
\begin{enumerate}[label=(\arabic*)]
    \item The Monge-Ampère measure of each $\phi_j$ is the restriction of a $\Gamma$-periodic measure $\mu_j$;
    \item $\phi_j \to \phi$ and $v_j \to v$ locally uniformly in $U$ as $j \to \infty$;
    \item $L_{\phi_j} v_j \geq 0$ in $U$  (resp. $L_{\phi_j} v_j \leq 0$ in $U$) holds in the classical sense.
\end{enumerate}
\end{Definition}

\begin{Remark} \label{rem:approximated construction}
Definition~\ref{def:approximated solution} is specifically designed for $\phi\in\E_{\Lambda}$. Indeed, for $\phi\in\E_{\Lambda}$, we define $\mu_j = \mu \ast \eta_j$, where $\{\eta_j\}$ is a sequence of strictly positive symmetric mollifiers with exponential decay (e.g., rescaled Gaussian kernels), and the approximating potentials $\phi_j$ are constructed by solving the Dirichlet problem
\begin{equation}\label{eq:approximated phi}
\det D^2 \phi_j= \mu_j  \quad  \text{in } \Omega, 
\quad \phi_j = \phi \quad \text{on } \partial \Omega.
\end{equation}
This construction ensures that each $\mu_j$ inherits the $\Gamma$-periodicity of $\mu$, and guarantees that $\phi_j\in C^2(\Omega)$ satisfying $D^2\phi_j>0$ in $\Omega$.
On any fixed $U'\Subset U$, we may assume $v_j\ge0$ (resp. $v_j\le0$) whenever $v\ge0$ (resp. $v\le0$), by adding constants tending to zero. We shall make such adjustments without further mention.
\end{Remark}

In the below, we start to prove truncated weak Harnack inequalities for  $\Gamma$-approximated supersolutions and subsolutions in large sections.

\begin{Lemma}\label{lem:Theorem 7.3.1 modify}
For any $\delta > 0$, there exist positive constants $\Lambda_4(n,\delta)$ and $M_1(n,\delta)$ with the following property: Let $\phi\in\E_\Lambda$ with $\Lambda\ge\Lambda_4$. If a section $S_h(y) \subset \Omega$ satisfies $\kappa(S_{h}(y)) > \Lambda_4$ and $v \in C(\Omega)$ is a nonnegative function satisfying $L_\phi v\le 0 $ in $S_{h}(y)$ with
\[
\inf_{(1-\delta)S_{h}(y)} v \leq 1,
\]
then
\[
\frac{\mu \left(\{v<M_1\} \cap S_h (y)  \right)}{\mu \left( S_{h}(y) \right)} \geq (1-\delta) .
\]
\end{Lemma}
\begin{proof}
We first consider the classical setting where $\phi\in C^2(\Omega)$ satisfying $D^2\phi>0$ in $\Omega$, and $v \in C^2(\Omega)$ is a classical nonnegative supersolution.
Without loss of generality, we may assume $S_{h}(y)$ is normalized (see, e.g., Steps 1 \& 2 of the proof of  \cite[Theorem 7.3.1]{gutierrez2016monge}), that is, $B_1 \subset S_{h}(y) \subset B_n$. For simplicity, we denote $S_{h}(y)=\Omega$.
Let 
\[
\omega= \min \left\{\frac{v+\alpha \phi }{\alpha},0  \right\},
\]
where $\alpha=\tau^{-1}$ with $\tau=\tau(n,\delta,\delta)$ being the one in Lemma \ref{lem:convex-comparison}. 
Since $v \geq 0$, we have $\omega \geq \phi$.
Let $w$ be the convex envelope of $\omega$ in $\Omega$, and 
\[
\C(\omega)=\{ x\in \Omega:\; w(x)=\omega(x) \}
\]
be the set of contact points. Then, we have $w \leq \omega$, and
\[
\mu_w = 0 \quad  \text{in } \Omega \setminus \C(\omega), \quad \text{and }  \mu_w\leq \mu_\phi=\mu \quad \text{in } \C(\omega),
\]
where the last inequality follows from the calculation in Step 7 of the proof of  \cite[Theorem 7.3.1]{gutierrez2016monge}. Note that 
\[
\inf_{(1-\delta)\Omega} (w-\phi) \leq  \inf_{(1-\delta)\Omega} \frac{v}{\alpha} \leq \frac{1}{\alpha}=\tau.
\]
Lemma \ref{lem:convex-comparison} immediately yields $\mu(\Omega \setminus \C(\omega)) \leq \delta$ when $\kappa(S_{h}(y)) >\Lambda_4(n,\delta)$ chosen sufficiently large. We complete the proof by noting that
\[
v= \alpha(w-\phi) =\alpha(\omega-\phi)\leq  \alpha\|\phi\|_{L^\infty}\leq M_0(n,\delta)\quad \text{in } \C(\omega),
\] 
where we used the Alexandrov estimate in the last inequality.

In the general case, let $\{\phi_j\}$ and $\{v_j\}$ be $\Gamma$-approximating sequences for $\phi$ and $v$ in Definition \ref{def:approximated solution}, respectively. 
Choose $y_j$ such that $\nabla\phi_j(y_j)=p$, where $p$ defines $S_h^\phi(y)$.
Then for each $j$, enlarging $M_0(n,\delta)$ if necessary, we still have:
\[
\mu_j \Big(  \{ v_j \leq M_0 \inf_{(1-\delta/2)S_h^{\phi_j}(y_j)} v_j  \} \cap S_h^{\phi_j}(y_j) \Big) \geq (1-\delta) \mu_j (S_h^{\phi_j}(y_j) )
\]
for any $S_h^{\phi_j}(y_j)\subset \Omega$ satisfying $\kappa(S_h^{\phi_j}(y_j)) > \Lambda_4(n,\delta)$.
In view of the weak-* convergence $\mu_j \to \mu$, and the locally uniform convergence $\phi_j \to \phi$ and $v_j \to v$, it follows that for all $S_h^{\phi}(y)\subset \Omega$ satisfying $\kappa(S_h^{\phi}(y)) > \Lambda_4(n,\delta)$, we have
\begin{align*}
\mu\left(\{v<4M_0\}\cap S_h^\phi(y)\right)
&\mathrel{\geq} \lim_{\epsilon\to0^+}
\mu\left(\overline{\{v<3M_0\}\cap S_{h-\epsilon}^\phi(y)}\right)\\
&\geq \lim_{\epsilon\to0^+}\limsup_{j\to\infty}
\mu_j\left(\overline{\{v<3M_0\}\cap S_{h-\epsilon}^\phi(y)}\right)\\
&\geq \lim_{\epsilon\to0^+}\liminf_{j\to\infty}
\mu_j\left(\{v_j<2M_0\}\cap S_{h-2\epsilon}^{\phi_j}(y_j)\right)\\
&\geq \liminf_{\epsilon\to0^+}\liminf_{j\to\infty}
\mu_j\Big(\{v_j\leq M_0\inf_{(1-\delta/2)S_{h-2\epsilon}^{\phi_j}(y_j)}v_j\}
\cap S_{h-2\epsilon}^{\phi_j}(y_j)\Big)\\
&\geq (1-\delta)\lim_{\epsilon\to0^+}\liminf_{j\to\infty}
\mu_j\left(S_{h-2\epsilon}^{\phi_j}(y_j)\right)\\
&\geq (1-\delta)\lim_{\epsilon\to0^+}\liminf_{j\to\infty}
\mu_j\left(S_{h-3\epsilon}^\phi(y)\right)\\
&\geq (1-\delta)\lim_{\epsilon\to0^+}
\mu\left(S_{h-3\epsilon}^\phi(y)\right)\\
&\mathrel{=}(1-\delta)\mu\left(S_h^\phi(y)\right).
\end{align*}
Therefore, the conclusion follows upon setting $M_1=4M_0$.
\end{proof}

We will use the following maximal-type operator for which we will establish its $L^\epsilon$ estimate in Theorem \ref{thm:harnack super} rather than for the solution itself.

\begin{Definition}\label{def:phi homogenization}
Let $\phi\in\E_\Lambda$. Fix $\lambda = \lambda(n) > 0$ sufficiently small such that $\sigma = 2^nC_1(n)\lambda$ satisfies $4^{n+1}\sigma < c(n)$, where $C_1(n)$ is the one in Lemma \ref{lem:covering Lemma} and $c(n)$ is the one in Proposition \ref{prop:basic}. For $\Lambda\ge \max\{\Lambda_0,\Lambda_1,\Lambda_2\}$, and any nonnegative Borel function $v$ on $\Omega$, we define its $\phi$-homogenization $M_{\phi,\Lambda}v(x):=M_{\phi,\lambda,\Lambda}v(x)$ in $\Omega$ through the following maximal-type operator:
\begin{equation}\label{eq:phi homogenization}
\begin{split}
M_{\phi,\Lambda}v(x)=  \sup\left\{ t\in\R:\; \exists \ \check{h}(x)<h<  \hat{h}(x)\ s.t. \
\frac{\mu ( \{v> t\} \cap S_h (x)  )}{\mu(S_h (x))} > \lambda \right\}.
\end{split}
\end{equation}  
The operator $M_{\phi,\Lambda}v$ is Borel measurable for every nonnegative Borel function $v$.
\end{Definition}
If $M_{\phi,\Lambda}v(y) > t>0$ at $y \in \frac{7}{8}\Omega$ and $\mu(\{v > t\}) < \lambda \mu(S_{ c(n)\hat{h}(y) }(y)) $, then
\begin{equation}\label{eq:def hy}
h_y :=h_{y,t}=\sup \left\{h<c(n)\hat{h}(y):\; \frac{\mu ( \{v> t\} \cap S_h (y)  )}{\mu(S_h (y))} \mathrel{>} \lambda \right\}
\end{equation}
is well-defined and satisfies
\[
\check{h}(y)<h_y<c(n)\hat{h}(y),\qquad
\lambda\le\frac{\mu(\{v>t\}\cap S_{h_y}(y))}{\mu(S_{h_y}(y))}\le 2^n\lambda.
\]

\begin{Lemma}\label{lem:Corollary 7.6.1}
Let $\phi\in\E_\Lambda$. For any $\delta > 0$, there exist constants $C(n) > 0$, $\Lambda_5(n,\delta) > 0$, and $L(n,\delta) > 0$ such that the following holds: For any $t > 0$, if $\Lambda \geq \Lambda_5$,  $\kappa(S_h(y)) > \Lambda$ and $S_{C^4h}(y) \subset \Omega$, and $v \in C(\Omega)$ is a nonnegative function satisfying $L_\phi v\le 0 $ in $S_{C^4h}(y)$ with
\[
\frac{\mu \left( \{ v> t \}  \cap S_{h}(y)\right)}{\mu \left( S_{h}(y) \right)} \geq \delta,
\] 
then we have  
\[
v(x) > t/L\quad \text{and} \quad M_{\phi,\Lambda}v(x) > t/L \quad \text{in } S_{Ch}(y).
\]
\end{Lemma}
\begin{proof}
Applying Lemma \ref{lem:Theorem 7.3.1 modify} to $S_{C^4h}(y)$ and observing that
\[
\frac{\mu \left( \{ v> t \}  \cap S_{C^4h}(y)\right)}{\mu \left( S_{C^4h}(y) \right)}\geq \frac{\mu \left( \{ v> t \}  \cap S_{h}(y)\right)}{\mu \left( S_{C^4h}(y) \right)} \geq c_1(n)\frac{\mu \left( \{ v> t \}  \cap S_{h}(y)\right)}{\mu \left( S_{h}(y) \right)}\geq c_1(n)\delta,
\]
we obtain that $v> t/L$ in $S_{C^3h}(y)$, where $L = M_1(n, c_1\delta)$. 
Furthermore, by the engulfing property of large sections, for every $x \in S_{Ch}(y)$ we have $S_h(y) \subset S_{C^2h}(x) \subset S_{C^3h}(y)$. This implies that  $\kappa(S_{C^2h}(x)) \geq \kappa(S_{h}(y)) > \Lambda$, $v>t/L$ in $S_{C^2h}(x)$ and $\check{h}(x)<C^2h <\hat{h}(x)$. Therefore, $M_{\phi,\Lambda}v(x) > t/L$. 
\end{proof}

\begin{Theorem}\label{thm:harnack super} 
Let $\phi\in\E_\Lambda$. Assume that $v \in C(\Omega)$ is a nonnegative function satisfying $L_\phi v\le 0 $ in $\Omega$.
Then there exist positive constants  $\Lambda_6(n)$, $\epsilon(n) $ and $C(n) $ such that if 
\[
\inf_{B_{1/4}} v \leq 1,
\]
then for $\Lambda\ge \Lambda_6$, we have
\begin{equation}\label{eq:harnack super} 
\mu(\{M_{\phi,\Lambda}v >  t\} \cap B_{1/2})< Ct^{-\epsilon} .
\end{equation} 
\end{Theorem}
 
\begin{proof}  
Recall that we chose $\sigma = 2^nC_1(n)\lambda$ sufficiently small so that it satisfies $4^{n+1}\sigma < c(n)$. For $\gamma(n) > 0$ sufficiently small, we choose $M>\max\{M_1(n,\lambda)L(n,\lambda),L(n,\lambda)^2\}$ sufficiently large (depending on $n$, $\lambda$ and $\gamma$).

Without loss of generality, we can assume $v>0$ in $\Omega$, since otherwise, we could consider the function $v+\tau$ for $\tau>0$ and send $\tau\to 0$ in the end. 
Now, we show that if $\check{h}(x_0) <h_0=h_0(x_0) < \hat{h}(x_0)$,
then
\begin{equation}\label{eq:harnack super b} 
\mu(\{M_{\phi,\Lambda}v >  t v(x_0)\} \cap S_{\sigma h_0}(x_0))< Ct^{-\epsilon} .
\end{equation}  
By scaling, we assume  $v(x_0)=1$ and the section $S_{h_0}(x_0)$ is normalized.
For notational simplicity, we denote $h_0=1$ and $x_0=0$, and the section $S_{h_0}(x_0)$ simply by $\Omega$.
We retain the pullback of the original maximal operator. By engulfing and Lemma \ref{lem:Theorem 7.3.1 modify}, every admissible section contributing at sufficiently high levels in $S_{2\sigma}(0)$ has its $C^4$-enlargement contained in $\Omega$. All level-set measures in the following local argument are taken in $\Omega$.

Consider the sets  
\[
D_{k}=\{M_{\phi,\Lambda}v >M^{k}\} \cap E_{k},\quad \text{where } E_{k}=S_{(1 +\sigma^{k/4n} )\sigma }(0)  \subset S_{2\sigma }(0) .
\]
By induction on $k$, we proceed to establish  
\begin{equation}\label{eq:dt decay}
\mu(D_{k}) \leq \gamma(n) \sigma^{k}.
\end{equation}

By Lemma \ref{lem:Theorem 7.3.1 modify}, $\mu(\{v > M\})$ is sufficiently small provided $M$ and $\Lambda$ are sufficiently large. Consider a point $y \in \{ M_{\phi,\Lambda}v > M \} \cap S_{2\sigma  }(0)$. Let $h_y$ be defined as in \eqref{eq:def hy}. Since
\[
c h_y^{\frac{n}{2}} \leq \mu(S_{h_y}(y)) \leq \lambda^{-1} \mu(\{v > M\}), 
\]
by choosing $M$ large enough, we can ensure that $h_y$ is sufficiently small to guarantee 
\[
S_{C^4 h_y}(y) \subset S_{2C \sigma  }(0) \subset \Omega.
\]
Then, by Lemma \ref{lem:Corollary 7.6.1}, we conclude that $y \in \{ v > M/L \}$, which establishes the base case $k = 1$ since $\mu(\{ v > M/L \})$ is small when $M/L$ is chosen sufficiently large.
 
For the inductive step, assume \eqref{eq:dt decay} holds for some $k \geq 1$. 
Since $\mu(\{v > M\} )$ is small, for every $y \in D_{k+1}$, there exists a small $h_y > \check{h}(y)$  such that
\[
\lambda\le\frac{\mu \left( \{  v > M^{k+1}/L \} \cap S_{h_y}(y) \right)}{\mu \left( S_{h_y}(y) \right)} \le 2^n\lambda. 
\]    
As in the previous argument, by choosing $M$  sufficiently large, we can ensure that $h_y$ is small enough to guarantee $S_{h_y}(y) \subset E_1$.  
Applying Lemma \ref{lem:Corollary 7.6.1}, we obtain
\[
S_{h_y}(y) \subset\{ M_{\phi,\Lambda}v  > M^{k+1}/L^2\}.
\]

We claim that
\[
S_{h_y}(y) \subset E_k.
\]
If this were not the case, let $1\le l < k$ be the smallest integer such that
\[
S_{h_y}(y) \subset E_{l},\quad \text{but } S_{h_y}(y)\not \subset E_{l+1}.
\]
Proposition \ref{prop:basic} would imply that 
\[
h_y \geq c(n)\operatorname{diam}(S_{h_y}(y))^{4}\geq  c(n)\operatorname{dist}(E_{k+1},\partial E_{l+1})^{4} \geq c(n)\sigma^{l/n}.
\]
where the last inequality follows from the Lipschitz regularity of $\phi$.
This leads to
\[
\mu(D_l) \geq \mu( S_{h_y}(y)) \geq   c(n) h_y^{\frac{n}{2}} \geq c(n) \sigma^l,
\]
contradicting \eqref{eq:dt decay} when $\gamma$ is sufficiently small.

In conclusion, we have 
\[
S_{h_y}(y) \subset (\{ M_{\phi,\Lambda}v  > M^{k+1}/L^2\} \cap E_{k}) \subset D_k.
\]
Furthermore, Lemma \ref{lem:Corollary 7.6.1}, applied at level $M^{k+1}$ as in the base case, also implies $D_{k+1} \subset    \{  v > M^{k+1}/L  \}$, and it follows that
\[
\frac{\mu \left( D_{k+1} \cap S_{h_y}(y) \right)}{\mu \left( S_{h_y}(y) \right)} \leq \frac{\mu \left( \{  v > M^{k+1}/L  \} \cap S_{h_y}(y) \right)}{\mu \left( S_{h_y}(y) \right)} \le 2^n\lambda. 
\] 
Applying Lemma \ref{lem:covering Lemma} to the sets $A = D_{k+1}$ and $B = D_k$ yields
\[
\mu(D_{k+1}) \leq \sigma  \mu(D_{k}).
\] 
This completes the induction step for \eqref{eq:dt decay}, from which we conclude for $t\ge M$
\[
\mu (\{M_{\phi,\Lambda}v(x) >t\} \cap S_{\sigma}) \leq  \gamma\sigma^{-1} t^{-\epsilon}\quad \text{for } \epsilon=|\log_{M}\sigma|. 
\]
This completes the proof of \eqref{eq:harnack super b}.

Finally, we complete the proof by means of a standard covering argument.
By assumption, we have $\inf_{B_{1/4}} v \leq 1$. 
Choosing $M$ sufficiently large ensures that $\mu(\{ v > M \})$ becomes sufficiently small.
Define
\[
E=\{ v \leq M\}\cap B_{3/4}, \quad  \mathcal{F} = \{S_{\sigma \hat{h}(y)}(y)\}_{y \in E}.
\] 
For any point $z \in B_{1/2} \setminus \bigcup\mathcal{F}$, the engulfing property implies that $S_{c\sigma \hat{h}(z)}(z) \subset B_{3/4}\setminus E$.
It follows that
\[
\mu(\{ v > M\})  \geq  \mu(S_{c\sigma \hat{h}(z)}(z)) \geq c(n)>0,
\]
which, for $M$ sufficiently large, leads to a contradiction.
Therefore, $\mathcal{F}$ must form a cover of $B_{1/2}$.
The same argument shows that the sections $S_{c\sigma \hat{h}(y)}(y)$, $y\in E$, also cover $B_{1/2}$.
By applying Lemma \ref{lem:vitali covering Lemma} to these smaller sections, we extract a finite subcover
$\{S_{\sigma \hat{h}(y_k)}(y_k)\}_{k=1}^{k_0}$ of $B_{1/2}$ from $\mathcal{F}$
such that the sections $S_{c\sigma \hat{h}(y_k)}(y_k)$ are pairwise disjoint.

Since each $\mu(S_{c\sigma \hat{h}(y_k)}(y_k))$ has a uniform positive lower bound, we conclude that $k_0 \leq C(n)$.
Then, using inequality \eqref{eq:harnack super b}, we obtain 
\[
\mu(\{M_{\phi,\Lambda}v >  t\} \cap B_{1/2}) \leq  \sum_{k=1}^{k_0}\mu(\{M_{\phi,\Lambda}v >  t\} \cap S_{\sigma \hat{h}(y_k)}(y_k)) \leq  \sum_{k=1}^{k_0}C(t/M)^{-\epsilon} \leq C(n) t^{-\epsilon} 
\]
for a larger constant $C(n) > 0$.
This concludes the proof.
\end{proof}

It should be noted that subsolutions do not admit one-sided $L^{\infty}$ bounds in general, since the doubling condition may break down for small sections. 
As a substitute, we will implement the following dichotomy:
\begin{Theorem}\label{thm:harnack sub} 
Let $\phi\in\E_\Lambda$. Assume that $v \in C(\Omega)$ is a nonnegative function satisfying $L_\phi v\ge 0 $ in $\Omega$.
Then, there exist positive constants $\Lambda_7(n)$ and $\beta(n)$ such that for any $\epsilon > 0$, there is a constant $C_2(n, \epsilon)$ with the following property: If $\Lambda\ge \Lambda_7$ and
\[
\min\{\mu(\{v >  t\}  \cap B_{1/2}), \mu(\{M_{\phi,\Lambda}v >  t\} \cap B_{1/2})  \} \leq t^{-\epsilon}   \quad  \forall t>0, 
\]
then either
\[
\sup_{B_{1/4} } v  \leq C_2(n,\epsilon), 
\] 
or for $\kappa: =\kappa(B_1)$ that
\[
\sup_{B_{1/2} } v  \geq e^{\beta \kappa^{\frac{1}{3}} } \sup_{B_{1/4} }v.
\] 
\end{Theorem}
\begin{proof}
Let $M_1(n,\lambda)$ be the constant defined in Lemma \ref{lem:Theorem 7.3.1 modify}. Taking $\gamma(n,\epsilon) > 0$ sufficiently small, we denote
\[
\nu := \frac{M_1}{M_1 - \frac{1}{2}} > 1, \quad 
t_k = \min\{ \gamma^{-1}\nu^k M_1/2,  \kappa^{\frac{n}{2\epsilon}}\}, \quad \rho_{k}=C(n)t_k^{-\frac{2\epsilon}{n}}.
\]
Moreover, we assume that $\gamma(n,\epsilon)^{-1}$, $\Lambda(n)$, and $C(n)$ are chosen sufficiently large such that
\[
\inf_{B_{3/4}(0)} \hat{h}\gg \rho_k \geq C(n) \kappa^{-1}>\sup_{B_{3/4}(0)} \check{h} \quad  \text{and } \inf_{B_{1/2}(0)}\mu(S_{\rho_k} ) \geq c(n) \rho_{k}^{\frac{n}{2}} > 2t_k^{-\epsilon}.
\]

If $\sup_{B_{1/4}} v \leq \gamma^{-1}M_1$, then the proof is completed by setting $C_2 = \gamma^{-1}M_1$. If instead $\sup_{B_{1/4}} v > \gamma^{-1}M_1$, then after replacing $v$  with $\frac{\gamma^{-1}M_1 v}{\sup_{B_{1/4}} v}$ (still denoted as $v$), we may assume without loss of generality that
\[
\sup_{B_{1/4} } v  = \gamma^{-1} M_1 .
\] 
Let $y_1\in \overline B_{1/4} $ be such that $v(y_1)=\gamma^{-1} M_1 $. Through an inductive argument on $k \geq 2$, we will construct a finite sequence of points $\{y_k\}\subset B_{7/16}$ satisfying
\begin{equation}\label{eq:exp growth}
v(y_k)= \gamma^{-1}\nu^{k-1}M_1  \quad \text{and } y_{k} \in \overline{S_{\rho_{k-1}}(y_{k-1})}.
\end{equation}
Note that $\rho_k$ is sufficiently small by previous construction. It follows that $S_{\rho_k}(y_k) \subset B_{1/2}$.

Suppose, for contradiction, that \eqref{eq:exp growth} fails the first time for $k + 1$; that is,
\[
\sup_{S_{\rho_k}(y_k)} v < \gamma^{-1}\nu^k M_1.
\]
Let us consider
\[
w(x) = \frac{\gamma^{-1}\nu^k M_1 - v(x)}{\gamma^{-1}\nu^{k-1} (\nu-1) M_1}.
\]
Then, $w \geq 0$ is a supersolution in $S_{\rho_k}(y_k)$  with $w(y_k)=1$. Applying Lemma \ref{lem:Theorem 7.3.1 modify} for the set 
\[
E=\{ v >  \gamma^{-1}\nu^k M_1 / 2\} \cap S_{\rho_k}(y_k)= \{ w < M_1\}  \cap S_{\rho_k}(y_k)
\]
yields  
\begin{equation}\label{eq:contract1}
\mu (E) \geq (1-\lambda) \mu \left(S_{\rho_k}(y_k) \right)>2t_k^{-\epsilon}.
\end{equation}
Note that by the choice of $\rho_k$, $S_{\rho_k}(y_k)$ is a large section. By applying the engulfing property between large sections, we obtain for $x\in S_{\rho_k}(y_k) $ that $S_{\rho_k}(y_k)\subset S_{C\rho_k}(x)$ and 
\[
\mu (\{ v>  \gamma^{-1}\nu^k M_1 / 2\} \cap S_{C\rho_{k}}(x)) \geq \mu (E) \geq (1-\lambda) \mu \left(S_{\rho_k}(y_k) \right) > \lambda \mu(S_{C\rho_{k}}(x)).
\]
This  implies $x\in \{ M_{\phi,\Lambda}v>  \gamma^{-1}\nu^k M_1 / 2\}$, and thus,
\[
S_{\rho_k}(y_k) \subset \{ M_{\phi,\Lambda}v >  \gamma^{-1}\nu^k M_1 / 2\}.
\]
Therefore, 
\begin{equation}\label{eq:large level set}
\mu (\{ M_{\phi,\Lambda}v>  \gamma^{-1}\nu^k M_1 / 2\}\cap S_{\rho_k}(y_k))  \geq  \mu\left( S_{\rho_k}(y_k)  \right)  \geq 2t_k^{-\epsilon}.
\end{equation}
The estimates \eqref{eq:contract1} and \eqref{eq:large level set} contradict with our assumptions that
\[
\mu (\{ M_{\phi,\Lambda}v >  \gamma^{-1}\nu^k M_1 / 2\}\cap B_{1/2}) \leq  (\gamma^{-1}\nu^k M_1 / 2)^{-\epsilon} \leq t_k^{-\epsilon}
\]
or
\[
\mu (E )\leq \mu (\{ v>  \gamma^{-1}\nu^k M_1 / 2\}\cap B_{1/2})\leq t_k^{-\epsilon}.
\]
Consequently, we must have $\sup_{S_{\rho_k}(y_k)} v \geq \gamma^{-1}\nu^k M_1$, which completes the induction step for $k+1$, provided that $y_{k+1}\in B_{1/2}$. 

From our choice of points ${y_k}$, we obtain the key estimate
\[
|y_{k+1}-y_k|\leq \operatorname{diam}(S_{4\rho_k}(y_k)) \leq C(n)\rho_k^{\frac{1}{3}},
\]
where we used Proposition \ref{prop:basic} in the second inequality. When $\gamma > 0$ is sufficiently small, the induction process indicates that for all $t_k < \kappa^{n/(2\epsilon)}$,
\[
|y_{k+1}| \leq  \frac{1}{4}+C(n) \sum_{i\leq k}\rho_i^{\frac{1}{3}}<  \frac{1}{4}+C(n) \sum_{i\leq k}(\gamma^{-1}\nu^i M_1/2)^{-\frac{2\epsilon}{3n}}< \frac{3}{8}.
\]
Furthermore, at the first time that $t_j = \kappa^{n/(2\epsilon)}$, we have $\rho_j=\rho_{j+1}=\cdots=C\kappa^{-1}$. By repeating the induction argument for $m =\lfloor c(n)\kappa^{\frac{1}{3}} \rfloor$ additional steps, we obtain
\[
|y_{j+m}- y_{j}| \leq C(n) \sum_{i= j}^{i=j+m-1}\rho_i^{\frac{1}{3}} \leq C(n)m\kappa^{-\frac{1}{3}} <\frac{1}{16}.
\]
This yields a point $y_{j+m} \in B_{7/16}$ 
satisfying
\begin{equation}\label{eq:induction step size}
v(y_{j+m}) \geq \nu^m v(y_j)\geq  \nu^{c(n)\kappa^{\frac{1}{3}}} v(y_1)\geq  e^{\beta(n) \kappa^{\frac{1}{3}}} \sup_{B_{1/4} } v , 
\end{equation}
where $\beta(n)=c(n)\log \nu$. This completes the proof.
\end{proof}
 
The following lemma will be used in the proof of Theorem \ref{thm:harnack}.
\begin{Lemma}\label{lem:critial density} 
Let $\phi\in\E_\Lambda$. There exist positive constants  $\Lambda_8(n)$ and $C(n)$ such that if $\Lambda\geq \Lambda_8$, then for any convex set $S \subset \frac{7}{8}\Omega$ whose inradius is larger than $C(n)\Lambda^{\frac12}\kappa^{-\frac{1}{2}} $, 
and any measurable subset $E \subset S$ with $\mu(E) \ge \frac12\mu(S)$, we have
\[
\mu (\{M_{\phi,\Lambda}\chi_E=1\} \cap S) > \frac14 \mu(S).
\] 
\end{Lemma}
\begin{proof}
For every point $x \in S \subset \frac{7}{8}\Omega$, Proposition \ref{prop:basic} implies that $\check{h}(x) \leq C(n)\Lambda^{\frac32}\kappa^{-\frac{3}{2}}$, and the diameter bound
\[
\operatorname{diam}(S_{2\check{h}(x)}(x))  \leq C(n)\check{h}^{\frac{1}{3}} \leq  C(n)\Lambda^{\frac12}\kappa^{-\frac{1}{2}}.
\] Consequently, $S_{2\check{h}(x)}(x) \subset (2S)\cap\Omega$. Since  $\kappa(S)\ge \kappa(2S)/2\ge\kappa(S_{2\check{h}(x)}(x))/2\ge \Lambda/2$ is large, we also have $\mu(2S)\leq 3^n\mu(S)$.
Let us denote 
\[
A =E \setminus \{M_{\phi,\Lambda}\chi_E=1\}, \quad B= \bigcup_{y\in A}  S_{2\check{h}(y)}(y) \subset 2S.
\]
For $y \in A$, since $M_{\phi,\Lambda}\chi_E(y)<1$, we have
\[
\frac{\mu(A \cap S_{2\check{h}(y)}(y))}{\mu(S_{2\check{h}(y)}(y))} \leq \frac{\mu(E \cap S_{2\check{h}(y)}(y))}{\mu(S_{2\check{h}(y)}(y))}  \leq \lambda. 
\]
Combining this with Lemma \ref{lem:covering Lemma}, we obtain
\[
\mu(A) \leq C_1(n)\lambda \mu(B) \leq C_1(n)\lambda  \mu(2S) <  C_1 (n)3^n\lambda\mu(S)<\frac{1}{4}\mu(S).
\]
This concludes the proof.
\end{proof}

Finally, we reach the following Harnack inequality with a quantitatively controlled error term.
\begin{Theorem}\label{thm:harnack}
Let $(\phi,\Omega,\mu,\Gamma)\in \E_\Lambda$ and $(\tilde \phi,\widetilde\Omega,\mu,\Gamma)\in \E_\Lambda$. Assume that 
$v \in C(\Omega \cup \widetilde{\Omega})$ is a nonnegative function satisfies
\[
L_{\phi}v \leq 0   \quad  \text{in } \Omega, \quad L_{\tilde{\phi}}v \geq 0   \quad  \text{in }\widetilde{\Omega}
\] 
in the sense of Definition \ref{def:approximated solution}.
Then there exist positive constants $\beta(n)$, $\Lambda_9(n)$,  and $C(n)$ such that if $\Lambda>\Lambda_9$, then either 
\[
\sup_{B_{1/4}}v\leq C\inf_{B_{1/4}}v,
\]
or for $\kappa =\kappa(B_1)$ that
\[
\sup_{B_{1/2}} v  \geq e^{\beta \kappa^{\frac{1}{6}}}\sup_{B_{1/4}}v.
\] 
Consequently, we have
\[
\sup_{B_{1/4}}v \leq C\inf_{B_{1/4}}v+e^{-\beta \kappa^{\frac{1}{6}}}\sup_{B_{1/2}} v.
\]
\end{Theorem}
\begin{proof} 
Without loss of generality, we can assume $v>0$ in $\Omega$. Then, we can further suppose $\inf_{B_{1/4}}v=1$. Applying Theorem \ref{thm:harnack super}, we derive that \eqref{eq:harnack super} holds for a large $\Lambda$.
The subsequent argument follows the same framework as the proof of Theorem \ref{thm:harnack sub}, with the following key modifications:

All sections $S_{\rho_{k-1}}$ are replaced by $S_{\rho_{k-1}}^{\tilde{\phi}}$. Consequently, in the growth condition \eqref{eq:exp growth}, the finite sequence ${y_k} \subset B_{1/2}$ to be constructed now satisfy $y_k \in \overline{S_{\rho_{k-1}}^{\tilde{\phi}}(y_{k-1})}$.
Similar to the proof of \eqref{eq:contract1},  we obtain
\[
\mu (\{ v>  \gamma^{-1}\nu^k M_1 / 2\} \cap S_{\rho_k}^{\tilde{\phi}}(y_k)\}) \geq (1-\lambda) \mu \left(S_{\rho_k}^{\tilde{\phi}}(y_k) \right) \geq \frac12 \mu \left(S_{\rho_k}^{\tilde{\phi}}(y_k) \right).
\] 
By Proposition \ref{prop:basic}, assuming $\rho_k \geq C \kappa^{-\frac{3}{4}}$, we conclude that the inradius of $S_{\rho_k}^{\tilde{\phi}}(x_k)$ is greater than $C \kappa^{-\frac{1}{2}}$.
In this way, we may still deduce from Lemma \ref{lem:critial density} the following lower bound, which serves as a replacement for estimate \eqref{eq:large level set}:
\[
\mu\left( \{ M_{\phi,\Lambda}v > \gamma^{-1}\nu^k M_1 / 2 \} \cap S_{\rho_k}^{\tilde{\phi}}(x_k) \right) \geq \frac14 \mu\left( S_{\rho_k}^{\tilde{\phi}}(x_k)  \right) \geq c(n)\rho_k^{\frac{n}{2}} .
\]  
To contradict the growth behavior \eqref{eq:harnack super}, and in view of the above, we now define the parameters $t_k$ and $\rho_k$ as follows:
\[
t_k = \min\{ \gamma^{-1}\nu^k M_1/2,  \kappa^{\frac{n}{4\epsilon}} \}, \quad \rho_{k}=C(n)t_k^{-\frac{2\epsilon}{n}} >\kappa^{-\frac{1}{2}}>C\kappa^{-\frac{3}{4}}.
\]
With these modifications, the total number of steps in our induction process becomes of order $\kappa^{\frac{1}{6}}$ in \eqref{eq:induction step size}. This accounts for the $\kappa^{\frac{1}{6}}$ factor appearing in the theorem's statement. 
\end{proof}

The above Harnack inequality leads to the following estimates.
\begin{Proposition}\label{prop:holder almost}
Suppose all the assumptions in Theorem \ref{thm:harnack}.
Then there exist positive constants $\alpha(n) \in (0,1)$, $\Lambda_{10}(n)$, and $C(n)$ such that if $\Lambda\ge\Lambda_{10}$ then for $\kappa = \kappa(B_1)$,
\[
|v(x) - v(y)| \leq C \left(|x - y|^{\alpha} + \kappa^{-\frac{\alpha}{4}}\right)\|v\|_{L^{\infty}(B_{1/2})} , \quad \forall\ x,y \in B_{1/4}.
\]
\end{Proposition}
\begin{proof}
For simplicity, we assume that $\|v\|_{L^{\infty}(B_{1/2})} = 1$ and only provide the proof at $y=0$.
Set $\hat h(0):=\min\{\hat h^\phi(0),\hat h^{\tilde\phi}(0)\}$.
By Proposition~\ref{prop:basic} (i) and (ii), for $\Lambda_{10}$ sufficiently large and $\kappa^{-\frac{3}{4}}\leq h\leq\hat h(0)$,
the sections $S_h^\phi(0)$ and $S_h^{\tilde{\phi}}(0)$ are comparable to $B_{h^{1/2}}(0)$, and
\[
c\kappa h^{1/2}\leq\kappa(S_h^\phi(0))\leq C\kappa h^{1/2},\qquad
c\kappa h^{1/2}\leq\kappa(S_h^{\tilde\phi}(0))\leq C\kappa h^{1/2}.
\]

Let $0<\rho(n)\leq\hat h(0)$ be small enough that both sections of height $\rho$ lie in $B_{1/2}$,
and let $l=\lfloor \frac{3}{4}\log_{1/\rho}\kappa \rfloor $. For $k=1,2,\cdots,  l$, we denote 
\[
S_k=S_{\rho^k}^\phi(0),\quad \text{and } M_k = \sup_{S_k} v, \quad m_k = \inf_{S_k} v.
\]
Noting that
\[
L_{\phi} v \leq 0 \quad \text{in } \Omega, \quad L_{\tilde{\phi}} v \geq 0 \quad \text{in }\widetilde{\Omega},
\]
a standard argument applying the Harnack inequality in Theorem \ref{thm:harnack}  to  $M_k - v$ and $v - m_k$ on $S_{c\rho^k}^{\phi}(0)$ and $S_{c\rho^k}^{\tilde\phi}(0)$, both contained in $S_k$ for a sufficiently small dimensional constant $c>0$, yields the relation
\[
M_{k+1} -m_{k+1}  \leq \max\{ (1-c)(M_{k}- m_{k}), 2e^{-\beta (\kappa(S_k) )^{\frac{1}{6}}}\}  
\]
for some $c(n) > 0$. Through an standard iterative process over $1\leq k \leq l$, we obtain the $\alpha$-H\"older continuity of $v$ at $0$, up to a small error term:
\[
2e^{-\beta (\kappa(S_l) )^{\frac{1}{6}}} +C\left(\operatorname{diam}(S_l)\right)^{\alpha}\leq e^{-c\beta\kappa^{\frac{1}{12}}} +C\rho^{\frac{\alpha l}{3}} \leq C\kappa^{-\frac{\alpha}{4}}
\]
for some small constant $\alpha(n)>0$ depending only on the dimension $n$.
\end{proof}

\section{Semiconcavity of global solutions}\label{sec:semiconvavity}

Let $\phi$ be a convex function in $\R^n$ whose Monge-Amp\`ere measure $\mu$ is $\Gamma_{\{\varepsilon_1, \cdots, \varepsilon_n\}}$-periodic. We define the  second incremental quotient associated with $\Gamma$ as
\[
\Delta^2_i\phi (x):=\Delta^2_{\varepsilon_i}\phi (x)= \frac{\phi(x+\varepsilon_i)+\phi(x-\varepsilon_i)-2\phi(x)}{|\varepsilon_i|^2} \geq 0.
\]

Note that $\Delta^2_i\phi$ is not affine invariant. Consider the rescaled function  
\[
\tilde{\phi}(x)=\frac{\phi(A  x)}{h}
\]
for some matrix $A \in \mathcal{S}_+^{n\times n}$ and $h>0$. Then $ \mu_{\tilde{\phi}}$ is $\Gamma_{\{A^{-1}\varepsilon_1, \cdots, A^{-1}\varepsilon_n\}}$-periodic, and we have the transformation law:
\begin{equation}\label{eq:transformation law}
\Delta^2_i\tilde{\phi} (x)  :=\Delta^2_{A^{-1}\varepsilon_i}\tilde{\phi} (x) =\frac{|\varepsilon_i|^2 }{h|A^{-1} \varepsilon_i|^2}  \Delta^2_{\varepsilon_i}\phi(A  x). 
\end{equation}

\begin{Lemma}\label{lem:no lay}
Let $\phi$ be a convex function in $\R^n$ whose Monge-Amp\`ere measure $\mu$ is $\Gamma_{\{\varepsilon_1, \cdots, \varepsilon_n\}}$-periodic. If $\mu \not\equiv 0$, then the graph of $\phi$ contains no line segments of infinite length (rays). Consequently, any section of $\phi$ is a bounded set.
\end{Lemma}
\begin{proof}
Suppose, to the contrary, that the graph of $\phi$ contains a ray. By subtracting an appropriate supporting linear function and applying a suitable affine transformation, we may assume that $\phi \geq 0$, $\phi = 0$ on $\{t e_n : t \geq 0\}$, and that $B_1(0)$ (and hence each $B_1(t e_n)$) contains a fundamental domain of the periodicity. Then, by the convexity of $\phi$, we have that $\phi$ is bounded on $B_2'(0) \times [1, \infty)$, where $B_2'(0)$ is the ball of radius $2$ in $\R^{n-1}$ centered at origin, and $\|\phi\|_{L^{\infty}(B_2'(0) \times [1, \infty))}\le \|\phi(\cdot,1)\|_{L^{\infty}(B_2'(0))}$.
Consequently,
\[
\partial \phi \left(B_1'(0) \times [2, \infty)\right) \subset \left\{ y : |y| \leq 2\|\phi\|_{L^{\infty}(B_2'(0) \times [1, \infty))} \right\}\subset \left\{ y : |y| \leq 2\|\phi(\cdot,1)\|_{L^{\infty}(B_2'(0))} \right\}.
\]
This contradicts the fact that
\[
\mu (B_1'(0) \times [2, \infty)) \geq \sum_{k \geq 2} \mu(B_1(2k e_n)) = \infty.
\]
\end{proof}

\begin{Lemma}\label{lem:approximated equation}
Let $\phi$ be a convex function in $\R^n$ whose Monge-Amp\`ere measure $\mu$ is $\Gamma_{\{\varepsilon_1, \cdots, \varepsilon_n\}}$-periodic. If $\mu \not\equiv 0$, then in any section $\Omega$ of $\phi$, the inequality
\[
L_{\phi} (\Delta^2_{i} \phi) \geq 0
\]
holds in the sense of Definition \ref{def:approximated solution}.
\end{Lemma}
\begin{proof}
Let $\{\mu_j\}$ and $\{\phi_j\}$ be the approximating sequences constructed in Remark \ref{rem:approximated construction} on a domain significantly larger than $\Omega$. The smooth functions $\Delta^2_{i} \phi_j$ converge locally uniformly to $\Delta^2_{i} \phi$ as $j \to \infty$.  Moreover, we recall that as in \cite[Lemma 2.2]{caffarelli2004liouville}, the concavity of the operator $F(D^2 w) =(\det D^2 w)^{\frac{1}{n}} $ on the space of positive definite matrices implies that 
\[
F_{kl}(D^2\phi_j(x)) D_{kl} \left[ \phi_j(x+\varepsilon_i) + \phi_j(x-\varepsilon_i) - 2\phi_j(x) \right] \geq 0,
\]
where $F_{kl}(D^2\phi_j) = \frac{\partial F}{\partial w_{kl}}(D^2\phi_j) = \frac{1}{n} (\det D^2\phi_j)^{\frac{1-n}{n}} (\Phi_j)^{kl}$, and $(\Phi_j)^{kl}$ denotes the cofactor matrix of $D^2 \phi_j$. 
This directly yields $L_{\phi_j} (\Delta^2_i \phi_j) \geq 0$, which verifies the third condition in Definition \ref{def:approximated solution}.
\end{proof}

Motivated by the results in \cite{caffarelli2004liouville} and \cite{li2019monge}, we establish the following semiconcavity estimate for global solutions to \eqref{eq:periodic equation ma}. 

\begin{Proposition}\label{prop:semiconcave}
Let $u$ be a convex solution to \eqref{eq:periodic equation ma} with $\mu$ being a $\mathbb{Z}^n$-periodic Borel measure (where $\mathbb{Z}^n := \Gamma_{\{e_1,\cdots,e_n\}}$) satisfying $\mu(D(\mathbb{Z}^n)) = 1$. 
Then there exists a unique compatible matrix $A$ satisfying:
\[
\sup_{ \R^n}\Delta_z^2  u=\frac{z^{\top}Az}{|z|^2} ,\quad \forall z \in \mathbb{Z}^n\setminus\{0\}.
\]
\end{Proposition} 
\begin{proof}
Without loss of generality, we may assume $u(0) = 0$ and $u \geq 0$. 
Let us define $S_h := \{x \in \mathbb{R}^n : u(x) < h\}$, and select $\Lambda_u>0$ such that the section $S_{\Lambda_u}$ satisfies $\kappa_{\mathbb{Z}^n}(S_{\Lambda_u}) \geq \Lambda_0^2$. 
Applying Proposition \ref{prop:quadratic-asymptotics}, there exists a compartiablity matrix $A\in \mathcal{S}_+^{n\times n}$ such that the quadratic function 
\[
P_A(x)=\frac{1}{2} x^{\top}Ax
\]
satisfies
\begin{equation}\label{eq:asymquadra}
\left|u(x)-P_A(x)\right|
\le C(1+|x|)^{2-\sigma}
\quad \text{for all } x \in \R^n,
\end{equation}
where $C$ and $\sigma$ are some positive constants. 
Let us consider the normalized function
\[
u_{h}(x) =\frac{u(T_h x)}{(\det T_h)^{\frac{2}{n}}} \quad \text{with } \quad T_hx = h^{\frac{1}{2}}x,
\]  
and let $\mu_h$ denotes the Monge--Ampere measure of $u_h$. Denote $\Gamma_h := \Gamma_{\{T_h^{-1}e_1, \dots, T_h^{-1}e_n\}}$.
The transformation law \eqref{eq:transformation law} then yields the following estimates at corresponding points 
\begin{equation}\label{eq:transformh}
\Delta^2_{i,h} u_h:=\Delta^2_{T_h^{-1} e_i} u_h=  \Delta^2_{e_i} u.
\end{equation}
Moreover, as established in Lemma \ref{lem:approximated equation}, the inequality 
\[
L_{u_h} (\Delta^2_{i,h} u_h) \geq 0
\]
holds in the sense of Definition \ref{def:approximated solution}.

\textbf{Step 1.} 
Let $c_2>c_1>0$. For any  open convex $E$ satisfying $B_{c_1}(0) \subset E \subset B_{c_2}(0)$, we are going to show that
\begin{equation}\label{eq:w21 estimate}
\int_{E} \Delta^2_{i,h}u_h(x) \ud \mu_h= \mu_h(E)\Delta^2_{i,h} P_A(0) +O(h^{-\frac{\sigma}{4}})\quad \text{as } h \to \infty,
\end{equation}
where the term $O(h^{-\frac{\sigma}{4}})$ depends only on $n,c_1,c_2,A$ and $u$, and 
\[
\Delta^2_{i,h} P_A(0)=\Delta^2_{T_h^{-1}e_i} P_A(0)=e_i^{\top}  A e_i.
\]

For sufficiently large $h$, it follows from \eqref{eq:asymquadra} and Lemma \ref{lem:uniformly converge implies gradient converge} that
\begin{equation}\label{eq:lipest}
\|u_h- P_A \|_{L^{\infty}(B_{4c_2})} \leq Ch^{-\frac{\sigma}{2}}\  \quad \text{and}\quad  \|u_h-P_A \|_{Lip(B_{2c_2})} \leq Ch^{-\frac{\sigma}{4}}.
\end{equation}
For fixed $h$ and $i$, let us consider the rescaled direction $e = T_h^{-1}e_i$, and  the directional accessibility set
\[
E^{+}= E+\{te:-1< t< 1\}.
\]  
For $x\in E^{+}$, we consider 
\[
\nu(x) = 
\begin{cases} 
0 & \text{if } x\in E,\\
-e & \text{if } x\not\in E\mbox{ and if }\exists \, t > 0 \text{ such that } x + t e \in E \quad \text{(``left''-accessible)}, \\
e & \text{if }  x\not\in E\mbox{ and if } \exists \, t < 0 \text{ such that } x + t e \in E \quad \text{(``right''-accessible)}.
\end{cases}
\]
Denote $F$ as the projection of $E$ onto the hyperplane perpendicular to $e$. 
We have
\begin{align*}
\mu_h (E^{+} \setminus E)= \mu_h (F\times (-e,e])\le C(n, c_2)|e|,
\end{align*}
where we used that $\mu_h$ is periodic in $e$ in the first equality, and used $E \subset B_{c_2}(0)$ and $\mu_h(D\Gamma_h)=|D\Gamma_h|$ in the last inequality. 
Then,
\begin{align*}
\int_{E} \Delta^2_{i,h}u_h(x) \ud \mu_h
&= \int_{ E^{+} \setminus E } \frac{u_h(x)-u_h(x-\nu)}{|e|^2} \ud \mu_h \\
&= \int_{ E^{+} \setminus E } \frac{P_A(x)-P_A(x-\nu)}{|e|^2}  \ud \mu_h +O(h^{-\frac{\sigma}{4}})\\
&= \int_{E} \Delta^2_{i,h} P_A(x) \ud \mu_h +O(h^{-\frac{\sigma}{4}})\\
&=\mu_h(E)\Delta^2_{i,h} P_A(0) +O(h^{-\frac{\sigma}{4}}),
\end{align*}
where we used \eqref{eq:lipest} in the second equality.

\textbf{Step 2.}
From \eqref{eq:w21 estimate}, we derive for all large $h$ that
\[
\mu_h(\{\Delta^2_{i,h} u_h >  t\}  \cap B_{1/2})\leq  C t^{-1}   \quad  \forall t>0. 
\]
Applying Theorem \ref{thm:harnack sub} to $L_{u_h}  \Delta^2_{i,h} u_h  \geq 0$ yields either
\begin{equation}\label{eq:w2infty bound}
\sup_{B_{1/4} } \Delta^2_{i,h} u_h  \leq C_2,
\end{equation}
or, for $\kappa:=\kappa_{\Gamma_h}(B_1) \geq ch^{\frac{1}{2}}$,
\[
\begin{split}
\sup_{B_{1/2} } \Delta^2_{i,h} u_h &\geq e^{\beta(n,A) \kappa^{\frac{1}{3}} } \sup_{B_{1/4} }\Delta^2_{i,h} u_h \\
&\geq ce^{\beta(n,A) \kappa^{\frac{1}{3}} } \int_{B_{1/4}}  \Delta^2_{i,h} u_h(x) \ud\mu_h \\ &
\geq ce^{\beta_1(n,A) h^{\frac{1}{6}} } .
\end{split}
\]

On the other hand, \eqref{eq:asymquadra} gives $\|u_h\|_{L^\infty(B_1)}\leq C$ for large $h$. Since $|T_h^{-1}e_i|^2=h^{-1}$,
\[
\sup_{B_{1/2}}\Delta^2_{i,h}u_h\leq 4h\|u_h\|_{L^\infty(B_1)}\leq Ch,
\]
which contradicts the exponential lower bound above for large $h$.
 
Hence, we are always in the first case, and the inequality \eqref{eq:w2infty bound} is satisfied. 

\textbf{Step 3.} 
For each fixed direction $i$, we denote
\[
\alpha=\sup_{\R^n}\Delta^2_{e_i} u=\sup_{\R^n}\Delta^2_{i,h} u_h,\quad \beta= e_i^{\top}  A e_i.
\] 
By sending $h\to\infty$, we obtain from  \eqref{eq:w21 estimate} that  $\alpha \geq \beta$.
We claim
\begin{equation}\label{eq:w2infty sup}
\alpha=\beta.
\end{equation}
Assume by contradiction that 
\[
\alpha = \beta + 4s
\] 
for some $s > 0$.

For every sufficiently large $h$, there exists a point $x_h \in B_c$ satisfying
\[
\alpha- \Delta^2_{i,h} u_h(x_h) \leq a_h,
\]
where $a_h \to 0 $ as $h \to \infty$. Define the normalized function $v = (\alpha - \Delta^2_{i,h} u_h)/a_h$ and set $\delta = s/[C(\beta + 2s)]$. Applying Lemma \ref{lem:Theorem 7.3.1 modify} yields with $\phi=u_h-1$ and $\Omega=S_{1,h}=\{u_h < 1\} $, we obtain that
\[
\mu_h ( \{v\leq M_1(n,\delta) \} \cap S_{1,h}  ) \geq (1-\delta) \mu_h(S_{1,h} ).
\]
For sufficiently small $a_h$, we consequently obtain
\[
\mu_h\left(\left\{\Delta^2_{i,h} u_h>\beta  +2 s\right\} \cap S_{1,h}\right)= \mu_h\left(\left\{v <\frac{2s}{a_h}\right\}  \cap S_{1,h}\right)\geq  (1-\delta) \mu_h (S_{1,h}).
\]
This leads to the integral lower bound
\[
\int_{S_{1,h}} \Delta^2_{i,h} u_h \ud \mu_h \geq (1-\delta) (\beta+2s) \mu_h (S_{1,h})\geq 
(\beta+s) \mu_h (S_{1,h}).
\]
This contradicts the uniform estimate \eqref{eq:w21 estimate} for sufficiently large $h$.

\textbf{Step 4.}  Now, for any $z \in \mathbb{Z}^n\setminus\{0\}$, let us select vectors $\mathbf{z}_1, \mathbf{z}_2, \dots, \mathbf{z}_n$ in $\mathbb{Z}^n$ with $\mathbf{z}_1 = z$ such that they form a coordinate system of $\mathbb{R}^n$. For simplicity, assume $\{\mathbf{z}_1,   \dots, \mathbf{z}_n\} = \{\T e_1, \dots, \T e_n\}$. Then, $\tilde{A}=|\det \T|^{-\frac{2}{n}} \T^{\top} A \T$ is the unique matrix corresponds to the blow-down limit of $\tilde{u}(x)= |\det \T|^{-\frac{2}{n}}u(\T x)$. Applying the  conclusion \eqref{eq:w2infty sup} to the function $\tilde{u}(y)$ implies that $\sup_{\R^n}\Delta^2_{e_1} \tilde{u}=e_1^\top\tilde{A}e_1$. Then it follows from \eqref{eq:transformation law} that $\sup_{\R^n}\Delta^2_z u=\frac{z^\top A z}{|z|^2}$. Finally, if there are two compatible matrices $A$ and $B$ such that  $z^\top A z=z^\top B z$ for all $z\in\mathbb{Z}^n$, then by taking $z=e_i$ and $z=e_i+e_j$, one derives that $A=B$.
\end{proof}

\section{Proof of the main Theorems}\label{sec:main theorem}

\begin{proof}[Proof of the existence of periodic solutions in Theorem \ref{thm:uniqueness within periodic}]
Let $v(x) + P(x)$ be a periodic solution to equation \eqref{eq:periodic equation ma} in the sense of Definition \ref{def:periodic solution}, where $v(x)$ is periodic and $P(x)=\frac{1}{2}x^{\top}Ax+bx+c$ is quadratic. By the convexity of $v + P$, we obtain from \cite[Lemma 2.2]{li1990existence} that
\begin{equation}\label{eq:c0 estimates}
|Dv(x)| \leq C(n)|A|.
\end{equation}
The existence of such periodic solutions was established in \cite{li1990existence} via the continuity method for cases where $\mu = f(x) \, \ud x$ with $f$ being smooth and strictly positive. 
This result naturally extends to a general $\mu$ by considering the approximating solutions $v_j$ to $\det D^2 (P+v_j) = \mu \ast \eta_j$, where $\{\eta_j\}$ is a sequence of strictly positive symmetric mollifiers with exponential decay  (e.g., rescaled Gaussian kernels). The convergence of $v_j$ to a solution is then guaranteed by the uniform $C^0$ estimate \eqref{eq:c0 estimates} and standard stability results for the Monge-Ampère equation.

The compatibility condition \eqref{eq:A compatible} can be verified as follows. First, consider the smooth case: if $v$ is smooth, then by direct computation as in \cite[Lemma 2.1]{li1990existence} we have
\[
\det A=\int_{\mathbb{T}^n} \det A \,\ud x= \int_{\mathbb{T}^n} \ud (v_1+A_{1j}x_j) \wedge \cdots \wedge \ud (v_n+A_{nj}x_j)=\mu(\mathbb{T}^n).
\] 
For the general case, we establish the identity via an approximation argument. First, we approximate $v+P$ by its smooth regularization $(v+P)*\eta_{\epsilon}$.
Since the Monge-Ampère measure is stable under the uniform convergence of solutions, $\det D^2((v+P)*\eta_{\epsilon})$ converges weakly to $\det D^2(v+P)$.
Passing to the limit then yields the compatibility condition between $A$ and $\mu$.
\end{proof}

\begin{proof}[Proof of Theorems \ref{thm:uniqueness within periodic} and \ref{thm:solutions are periodic}]  
Let $u$ be a global solution of \eqref{eq:periodic equation ma}. Without loss of generality, we assume that $\mu(\mathbb{T}^n)=1$. By Lemma \ref{lem:no lay}, $u$ cannot be linear on any ray. 
Therefore, there exists a sufficiently large integer $N$ such that $S_1(0) \subset B_N(0)$.
We now rescale $u$ by defining  $u(Nx)/N^2$. For simplicity, we continue to denote this rescaled function as $u$. Hence, by Proposition \ref{prop:semiconcave}, there exists a unique compatible matrix $A \in \mathcal{S}_+^{n \times n}$ such that 
\[
\sup_{ \R^n}\Delta_z^2  u=\frac{z^{\top}Az}{|z|^2},\quad \forall z \in \mathbb{Z}^n\setminus\{0\}.
\]
Let $\lambda=N$. Let $w$ be a solution of
\[
\det D^2 w=\mu_\lambda
\quad \text{in }\mathbb R^n,
\qquad
\mu_\lambda(E):=\lambda^{-n}\mu(\lambda E),
\]
such that $w(x)-\frac12 x^\top A x$ is $\lambda^{-1}\mathbb Z^n$-periodic. Then
\[
\Delta^2_z w=\frac{z^\top A z}{|z|^2}
\quad
\text{for all }z\in\mathbb Z^n\setminus\{0\} \subset \lambda^{-1}\mathbb Z^n .
\]
Let 
\[
v(x)=u(x)-w(x).
\]
By subtracting a suitable linear function, we now assume that $v(0)=0$ and 
\[
v(e_i)=v(-e_i),\quad 1\leq i\leq n.
\]
Noting that $\sup_{\R^n} \Delta_z^2 v = \sup_{\R^n} \Delta_z^2 (u - w) = 0$, we then find that
\[
v(\pm ke_i)\leq 0 \quad  \text{for all } k\in \mathbb{Z},\ 1\leq i\leq n.
\]

To simplify our discussion, let us set $A = \T^2$ and consider the functions
\[
u(\T^{-1}x),\quad v(\T^{-1}x),\quad w(\T^{-1}x)
\]
which we still denote as $u$, $v$, and $w$. Then the measure $\mu$ and the function $w-\frac12|x|^2$ become $\Gamma_{\{\varepsilon_1, \cdots, \varepsilon_n\}}$-periodic with $\varepsilon_i = \T e_i$. And we have 
\begin{equation}\label{eq:v convex midpoint}
\sup_{ \R^n}\Delta_z^2  u= \Delta_z^2  w=1 \quad \text{and } \sup_{ \R^n}\Delta_z^2  v=0, \quad \forall z \in \Gamma_{\{\varepsilon_1, \cdots, \varepsilon_n\}},
\end{equation}
and
\begin{equation}\label{eq:1/2}
\quad v(\pm k\varepsilon_i)\leq v(0)= 0 \quad \text{for all } k\in \mathbb{Z},\ 1\leq i\leq n.
\end{equation}
Moreover, since $\Delta_z^2  u\le 1$, for every $y$ with supporting linear function $\ell_y$ of $u$ at $y$, we also have 
\[
(u-\ell_y)(y+z)\le (u-\ell_y)(y+z)+(u-\ell_y)(y-z)\le |z|^2 \quad\mbox{for all } z\in\Gamma.
\]
This implies that for each section $S_h^u(y)$ with $h \geq C(n)$, we have $B_{h^{\frac{1}{2}}/2}(y) \subset S_h^u(y)$. Since $|S_h^u(y)| \leq C(n) h^{\frac{n}{2}}$ for large $h$, it follows that
\[
B_{h^{\frac{1}{2}}/2}(y) \subset S_h^u(y) \subset B_{C(n)h^{\frac{1}{2}}}(y) .
\]
Similar results hold for $w$ for large $h$,
\[
\quad B_{h^{\frac{1}{2}}/2}(y) \subset S_h^w(y) \subset B_{C(n)h^{\frac{1}{2}}}(y) .
\]

\textbf{Step 1.}  We claim that
\begin{equation}
\sup_{\R^n} v < \infty.
\end{equation}
The proof follows from the arguments in \cite[Lemma 2.9]{caffarelli2004liouville} and \cite[Lemma 4.5]{li2019monge}.
The constants below shall also depend on $A$ and the period $\{\varepsilon_1, \cdots, \varepsilon_n\}$.

Let us denote
\[
M_{R} =\sup_{B_{R} } v. 
\]
Suppose $v$ is not bounded above. Then we have $\lim_{R \to \infty} {M}_R = \infty$.
For large $R$,  by noting for each $x \in B_{3R/2}$, there exists $z \in  \Gamma$ such that $x-2z \in \left\{\sum_{i=1}^n t_i \varepsilon_i,\ -1 \leq t_i \leq 1 \right\}$. Using the fact that $\Delta_z^2 v \leq 0$, we derive $v(x) \leq 2v(x-z)- v(x-2z) \leq 2v(x-z)+C$ for $x-z \in B_{R}$, where $C>0$ depends on $A$. This yields 
\begin{equation}\label{eq:compare m}
M_{3R/2}\leq 2M_{R}+C\quad\mbox{for all large }R.
\end{equation}

For sufficiently large $R$, let us denote
\[
H_R(x) = \frac{M_{2R}-v(Rx)}{M_{2R}} \geq 0 \quad \text{in } B_2.
\] 
We first note that the inequalities
\[
L_w v \geq 0 \quad \text{and} \quad L_u v \leq 0
\]
hold in the sense of Definition \ref{def:approximated solution}. This is verified by the approximating sequences constructed in Remark \ref{rem:approximated construction}, where the approximants of $v = u - w$ are inherited directly from the respective approximations of the potentials $w$ and $u$.
By applying the Harnack inequality (Theorem \ref{thm:harnack}) to $H_R$ with $\phi=R^{-2}w(Rx)$ and $\tilde\phi=R^{-2}u(Rx)$, each shifted affinely to have zero boundary values on its height-$1$ section between $B_1$ and $B_{3/2}$, with noticing that $\kappa$ is comparable to $R$, we conclude that either:
\begin{equation}\label{eq:HR bounded}
\sup_{B_{1/4} } H_R  \leq C(n)\inf_{B_{1/4}}H_R ,
\end{equation}
or the exponential growth behavior that
\[
\sup_{B_{1/2} } H_R \geq   e^{\beta(n) R^{\frac{1}{6}}}\sup_{B_{1/4}}H_R \geq e^{\beta(n) R^{\frac{1}{6}}}H_R (0)=e^{\beta(n) R^{\frac{1}{6}}}.
\]
When $R$ is large, the second scenario contradicts the known polynomial growth bound $|v(x)| \leq 2|x|^2$ for large $x$.
Thus, we conclude for large $R$ that
\[
0 \leq H_R \leq C(n) \quad \text{in } B_{1/4}.  
\]
Moreover, applying Proposition \ref{prop:holder almost}, we obtain
\[
|H_R(x) - H_R(y)| \leq C\left( |x - y|^{\alpha} + R^{-\frac{\alpha}{4}}\right), \quad \forall\ x,y \in B_{1/8}.
\] 
Using the Arzel\`a-Ascoli theorem, there exists a sequence $\{R_j\}$ with $R_j\to\infty$ that  that $H_{R_j}$ converges uniformly to a H\"older continuous function $H$ in $B_{1/16}$ with 
$H(0)=1$.

From \eqref{eq:v convex midpoint} on $v$, we have 
\[
\Delta_z^2 H_R \geq 0,\quad \forall\ z \in \Gamma_{\{R^{-1}\varepsilon_1, \cdots, R^{-1}\varepsilon_n\}}.
\]
This, combined with the uniform convergence, implies that $H$ is a convex function.
Now, let $\ell(x)=p \cdot x+1$ be a supporting function of $H$ at $0$, for some $p \in \R^n$, that is,
\[
H(x) \geq \ell(x)=p \cdot x +1 \quad \text{in } B_{1/16}.
\]
We claim that 
\begin{equation}\label{eq:Hlinear}
H(x) = \ell(x) \quad \text{in } B_{1/128}.
\end{equation}
Indeed, by the uniform convergence, we may assume along a subsequence $R_j$ of $R$ (still denoted by $R$) that $\|H_R-H\|_{L^{\infty}(B_{1/16})} \leq \delta_R $, where $\delta_R \to 0$ as $R \to \infty$. Thus,
\[
H_R(x) \geq \ell(x)-\delta_R\quad \text{in } B_{1/16} .
\]
Applying the Harnack inequality (Theorem \ref{thm:harnack}) to the nonnegative function $H_R-\ell(x)+\delta_R$ in $B_{1/16}$ now yields
\[
\sup_{B_{1/128}} \left(H_R-\ell+\delta_R \right)\leq  C\delta_R+e^{-\beta(n) R^{\frac{1}{6}}} \sup_{B_{1/64}}  \left(H_R-\ell +\delta_R\right) \leq  C(n)(\delta_R+e^{-\beta(n) R^{\frac{1}{6}}}).
\]
This implies \eqref{eq:Hlinear} by sending $R\to\infty$.

Recalling \eqref{eq:1/2}, we have
\[
H_R\left(k\varepsilon_i/R\right) \geq H_R(0)=1,\quad k\in\mathbb Z,\quad 1\leq i\leq n.
\]
In particular, it follows that $H(t \varepsilon_i) \geq H(0)= 1$ for $|t \varepsilon_i| \leq \frac{1}{128}$, which, together with \eqref{eq:Hlinear}, implies
\[
H \equiv 1 \quad \text{in } B_{1/128}.
\]
However, by \eqref{eq:compare m}, we find
\[
\inf_{B_{1/128}} H_R= \frac{M_{2R}-M_{R/128}}{M_{2R}} \leq 1-c <1
\]
for a constant $c > 0$ independent of sufficiently large $R$. Hence, $\inf_{B_{1/128}} H \leq 1 - c$, which  contradicts to the previously established fact that $H \equiv 1$ in $B_{1/128}$.

\textbf{Step 2.} 
We now prove that $u$ must be a periodic solution in the sense of Definition \ref{def:periodic solution} and is unique up to an additive constant.

After adjusting by constants, we may assume the function $\tilde{v}:=-v=w-u$ satisfies 
\[
\tilde{v} \geq 0 \quad \text{and } \inf_{\R^n}\tilde{v}= 0.
\]
We prove $\tilde{v} \equiv 0$ by contradiction. Assuming $\tilde{v} \not\equiv 0$, say  $\tilde{v}(x_0) > 0$.
For sufficiently large $R$, we apply the Harnack inequality (Theorem \ref{thm:harnack}) to the rescaled function $\tilde{v}_R(x) := (C(n)R)^{-2}\tilde{v}(C(n)Rx)$ and obtain that either:
\begin{equation}\label{eq:v<C/j}
\sup_{B_{R/4} } \tilde{v}  \leq C(n)\inf_{B_{R/4}}\tilde{v},
\end{equation}
or the exponential growth behavior that
\[
\sup_{B_{R/2} } \tilde{v} \geq   e^{\beta(n) R^{\frac{1}{6}}}\sup_{B_{R/4}}\tilde{v} \geq e^{\beta(n) R^{\frac{1}{6}}}\tilde v(x_0).
\]
When $R$ is large, the second scenario contradicts the known polynomial growth bound $\tilde{v}(x) \leq 8|x|^2$ for large $x$.
Thus, we conclude from \eqref{eq:v<C/j} that
\[
\tilde{v}(x) \leq C \lim_{R\to \infty}\inf_{B_{R/4}} \tilde{v} = C \inf_{\R^n} \tilde{v}= 0 \quad \text{for all } x,
\]
which completes the proof.
\end{proof}

\section{An example: the Dirac comb}\label{sec:periodicdelta}
This section is devoted to a detailed analysis of the periodic Dirac measure case introduced in Example \ref{ex:periodicdelta}. We study global convex solutions of the Monge-Amp\`ere equation
\begin{equation}\label{eq:periodicdelta}
\det D^2 u =\sum_{z \in \mathbb{Z}^n} \delta_z  \quad \text{in } \mathbb{R}^n, 
\end{equation}
where the right-hand side is the $\mathbb{Z}^n$-periodic Dirac measure supported on the integer lattice.  

By Theorems~\ref{thm:uniqueness within periodic} and~\ref{thm:solutions are periodic}, any solution $u$
admits the decomposition
\[
u(x)=v(x)+\frac{1}{2} x^{\top}Ax+b\cdot x+c,
\]
where $v$ is $\mathbb{Z}^n$-periodic, $A\in\mathcal{S}_+^{n\times n}$ with $\det A=1$, $b\in\mathbb{R}^n$,
and $c\in\mathbb{R}$. Consequently, it suffices to construct and analyze solutions in this periodic form.

For simplicity, we assume below that $b=c=0$ and $u(0)=v(0)=0$. Then
\[
u(z)= \frac{1}{2}z^{\top}Az \quad \text{for all } z \in \mathbb{Z}^n.
\]

In one dimension, we have $A=1$, and 
\[
u(z+y)=\frac{1}{2} z^2 + \left(z+\frac{1}{2}\right)\cdot y,  \quad \text{for all } z \in \mathbb{Z}^n, \ y\in [0,1].
\] 
And  its subdifferential at each $z \in \mathbb{Z}^n$ satisfies $\partial u(z)=[z-\frac{1}{2},z+\frac{1}{2}]$.

In higher dimensions, the explicit form of $u$ depends on $A$ and is, in general, considerably more complicated.
We observe that the subdifferentials $\{\partial u(z)\}_{z\in\mathbb{Z}^n}$ are convex bodies and that they form a tiling of $\R^n$:
they cover $\R^n$, and distinct tiles have disjoint interiors.
Moreover, by Theorem \ref{thm:solutions are periodic}, any two tiles $\partial u(z)$ and $\partial u(\tilde z)$ are translates of each other.
Such convex bodies are necessarily centrally symmetric convex polytopes, usually referred to as \emph{parallelohedra}.

The systematic study of parallelohedra goes back to Fedorov \cite{fedorov1885elements}, Minkowski \cite{minkowski1897allgemeine}, Voronoi \cite{voronoi1908nouvelles}, and Delaunay (Delone) \cite{delaunay1929partition}. The seminal works of Minkowski \cite{minkowski1910geometrie} and Voronoi \cite{voronoi1908nouvelles} connected parallelohedra with the geometry of quadratic forms. 
In subsequent work, Venkov \cite{venkov1954aclass} and McMullen \cite{mcmullen1980convex} established structural criteria for a convex polytope to be a parallelohedron, and proved that any convex polytope which tiles space by translation admits a face-to-face lattice tiling.

Among parallelohedra, a distinguished special case arises from lattices: the Dirichlet-Voronoi polytope, defined as follows.

\begin{Definition}[Dirichlet-Voronoi polytope]
Let $n\ge 1$, and $\Gamma$ be an $n$-dimensional lattice.
The \emph{Dirichlet-Voronoi polytope} at $\xi\in\Gamma$ is denoted by $\operatorname{Vor}_{\Gamma}(\xi)$ and consists of all points $x\in\mathbb{R}^n$ that are closer to $\xi$ than to any other lattice point in $\Gamma$; that is,
\[
\operatorname{Vor}_{\Gamma}(\xi)
=  \{
x \in \R^n :\; |x-\xi|\le |x-\tilde{\xi}|, \ \forall  \tilde{\xi}\in \Gamma  \}.
\] 
Each $\operatorname{Vor}_{\Gamma}(\xi)$ is a convex polytope satisfying $|\operatorname{Vor}_{\Gamma}(\xi)| = |D(\Gamma)|$ and is a translate of $\operatorname{Vor}_{\Gamma}(0)$, namely
$\operatorname{Vor}_{\Gamma}(\xi)=\operatorname{Vor}_{\Gamma}(0)+\xi$.
Moreover, the family $\{\operatorname{Vor}_{\Gamma}(\xi)\}_{\xi\in\Gamma}$ tiles $\mathbb{R}^n$.
We denote by $\operatorname{vert}(\operatorname{Vor}_{\Gamma})$ the set of all vertices of the Voronoi tiling, i.e.,
\[
\operatorname{vert}(\operatorname{Vor}_{\Gamma})
=\bigcup_{\xi\in\Gamma}\operatorname{vert}\bigl(\operatorname{Vor}_{\Gamma}(\xi)\bigr),
\]
where $\operatorname{vert}(P)$ denotes the vertex set of a polytope $P$.
\end{Definition} 

A complete classification of parallelohedra is highly non-trivial.
One of the central conjectures in the theory is the Voronoi conjecture, posed by Voronoi in 1909, which relates this classification to Dirichlet-Voronoi polytopes of $n$-dimensional lattices.

\begin{conjecture}[Voronoi]
For every $n$-dimensional parallelohedron $P$, there exists a $n$-dimensional lattice $\Gamma$ and an affine transformation $A$ such that $A(P)$ is the Dirichlet-Voronoi polytope of $\Gamma$.
\end{conjecture}

The conjecture was proved by Voronoi \cite{voronoi1908nouvelles} for $n=2,3$, by Delaunay \cite{delaunay1929partition} for $n=4$, and more recently by Garber \cite{garber2025voronoi} for $n=5$.
In particular, the 110{,}244 combinatorial types of five-dimensional Dirichlet-Voronoi polytopes enumerated in \cite{sikiric2016complete} constitute the complete list of combinatorial types of five dimensional parallelohedra.

This Voronoi viewpoint  allows us to describe the tiling $\{\partial u(z)\}_{z\in\mathbb{Z}^n}$ for every solution $u$ of \eqref{eq:periodicdelta}.

\begin{Theorem}\label{thm:periodicdelta-tiling}
Let $u$ be a global convex solution of \eqref{eq:periodicdelta}. Let $A$ be as in \eqref{eq:periodic solution} and set $\Gamma:=A^{\frac{1}{2}}\mathbb{Z}^n$. 
Then the family $\{\partial u(z)\}_{z\in\mathbb{Z}^n}$ forms a tiling of $\mathbb{R}^n$ that is periodic under translations by $A\mathbb{Z}^n$, and each tile is a translate of $A^{\frac{1}{2}}\operatorname{Vor}_{\Gamma}(0)$.
\end{Theorem}

To obtain an explicit formula for $u$, we shall introduce the Delaunay decomposition associated with $\Gamma$.

\begin{Definition}[Delaunay decomposition]\label{def:delaunay}
Let $\Gamma\subset\mathbb{R}^n$ be an $n$-dimensional lattice.
A subset $S\subset \Gamma$ is called a \emph{Delaunay set} if there exists a (closed) ball $B$ such that
\[
S=\Gamma\cap \partial B
\quad\text{and}\quad
\operatorname{int}(B)\cap \Gamma=\emptyset.
\]
The \emph{Delaunay decomposition} (also called the \emph{Delaunay tessellation}) of $\Gamma$ is the collection
\[
\operatorname{Del}_\Gamma:=\{\operatorname{conv}(S): S\subset \Gamma \text{ is a Delaunay set}\},
\]
where $\operatorname{conv}(\cdot)$ denotes the convex hull.
Elements of $\operatorname{Del}_\Gamma$ are called \emph{Delaunay cells}.
\end{Definition}

The following duality originates from Delaunay's foundational work \cite{delaunay1934sphere}.
\begin{Proposition}[Voronoi-Delaunay duality]\label{rem:voronoi-delaunay-duality}
The Delaunay decomposition is dual to the Dirichlet-Voronoi decomposition of $\Gamma$:
a set $\{\xi_0,\dots,\xi_m\}\subset\Gamma$ defines a $k$-dimensional Delaunay cell if and only if   $\operatorname{Vor}_\Gamma(\xi_0),\dots,\operatorname{Vor}_\Gamma(\xi_m)$ intersect in a common$(n-k)$-dimensional face, and no other $\operatorname{Vor}_\Gamma(\xi)$ contains this intersection.
 
In particular, each vertex $p$ of the Voronoi tiling determines a Delaunay cell
\[
\operatorname{Del}_\Gamma(p):=\operatorname{conv}\{\xi\in\Gamma:\ p\in \operatorname{Vor}_\Gamma(\xi)\}.
\]
And we have $\Gamma\cap\partial B=\{\xi\in\Gamma:\ p\in \operatorname{Vor}_\Gamma(\xi)\}$ for some nonempty ball $B$ centered at $p$.
\end{Proposition}

In two dimensions, the Dirichlet-Voronoi polytope is either a centrally symmetric parallelogram or a centrally symmetric hexagon \cite{voronoi1908nouvelles}.

\begin{figure}[H]
    \centering
    \begin{minipage}{0.45\textwidth}
        \centering
\includegraphics[trim = 10mm 10mm 10mm 20mm, clip, width=6.8cm]{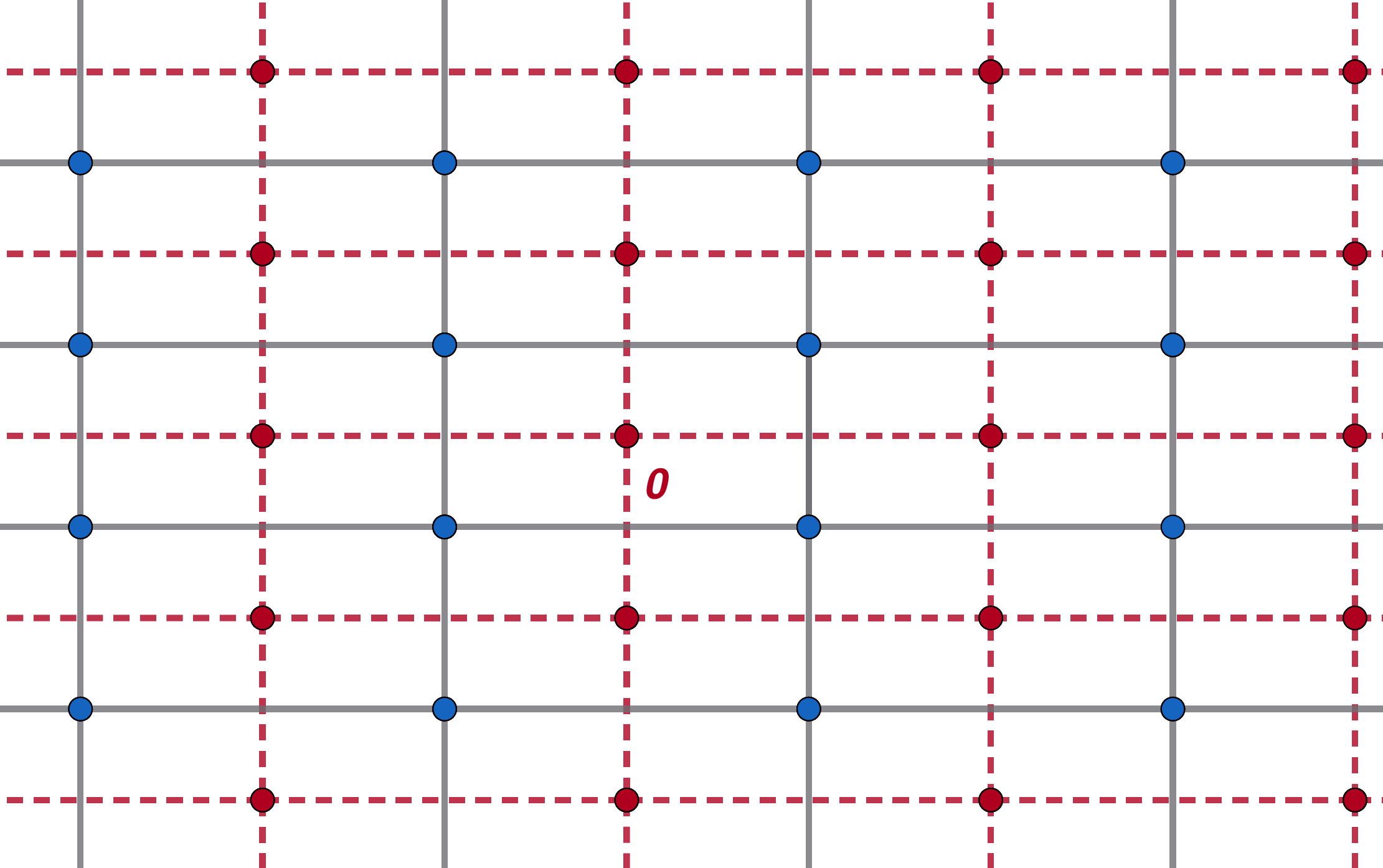}
        \label{fig:B2_left}
    \end{minipage}
    \hfill
    \begin{minipage}{0.45\textwidth}
        \centering
\includegraphics[trim = 20mm 10mm 10mm 20mm, clip, width=6.8cm]{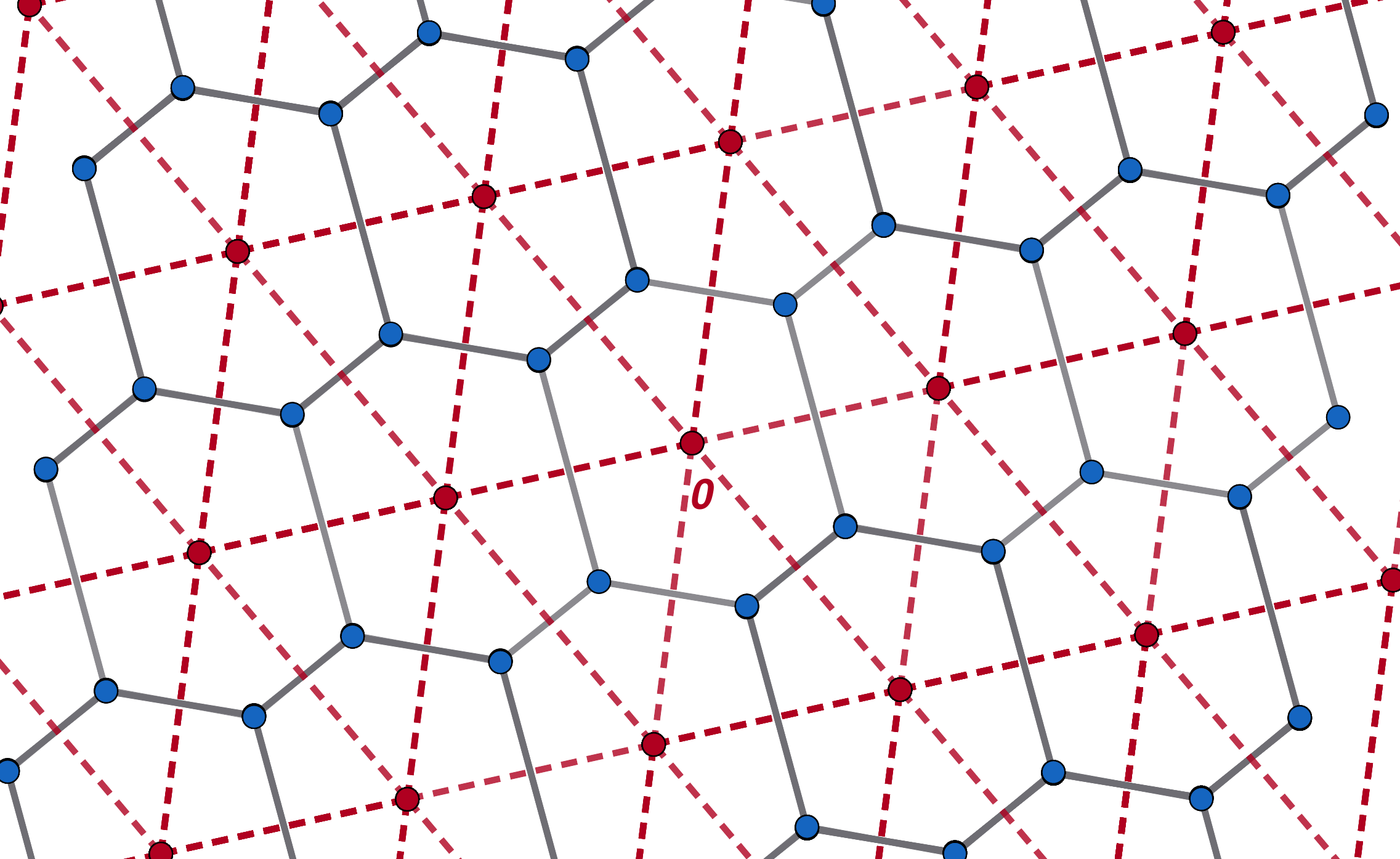} 
        \label{fig:B2_right}
    \end{minipage}
\caption{The Voronoi-Delaunay duality: the gray edges depict the location of Dirichlet-Voronoi polytopes, and the red edges depict the Delaunay decomposition of lattices.}
   \label{fig:both}
\end{figure}

\begin{Theorem}\label{thm:periodicdelta-formula}
Let $A\in\mathcal{S}_+^{n\times n}$ satisfy $\det A=1$, and set $\Gamma:=A^{\frac{1}{2}}\mathbb{Z}^n$.
For each $p\in\operatorname{vert}(\operatorname{Vor}_{\Gamma})$ and each $x\in\operatorname{Del}_{\Gamma}(p)$, define
\begin{equation}\label{eq:defudirac}
w(x):=\frac{1}{2}|\xi|^2+p\cdot (x-\xi),
\end{equation}
where $\xi\in\Gamma$ is any point such that $p\in \operatorname{Vor}_{\Gamma}(\xi)$.
Here, the admissible $\xi$ are precisely the vertices of $\operatorname{Del}_{\Gamma}(p)$, and the right-hand side of \eqref{eq:defudirac} is independent of the choice of $\xi$.

Then, up to an additive constant, the function $u(x):=w(A^{\frac{1}{2}}x)$ is the unique global convex solution of \eqref{eq:periodicdelta} with quadratic part $\frac{1}{2} x^{\top}Ax$.
\end{Theorem}

\begin{proof}[Proof of Theorems \ref{thm:periodicdelta-tiling} and \ref{thm:periodicdelta-formula}] 
For $\xi\in\Gamma$ and $x\in\operatorname{Vor}_{\Gamma}(\xi)$, we note from the definition of Dirichlet-Voronoi polytope that
\begin{align*}
\frac{1}{2}|\xi|^2+\xi \cdot (x-\xi)
=\sup_{\tilde{\xi}\in \Gamma} \left(\frac{1}{2}|\tilde{\xi}|^2+\tilde{\xi} \cdot (x-\tilde{\xi})\right).
\end{align*}
Motivated by this, we define
\[ 
w^*(x)=\sup_{\tilde{\xi}\in \Gamma} \left(\frac{1}{2}|\tilde{\xi}|^2+\tilde{\xi} \cdot (x-\tilde{\xi})\right)= \frac{1}{2}|\xi|^2+\xi \cdot (x-\xi), \quad  \xi \in \Gamma, \ x \in \operatorname{Vor}_{\Gamma}(\xi).
\]
Then $w^*$ is the supremum of a family of affine functions; in particular, $w^*$ is a well-defined convex function and is super-linear at infinity. 

Let us consider the Legendre transform $w^{**}$ of $w^*$. Then $w^{**}$ is a well-defined convex function on $\mathbb{R}^n$, and we have
\[
w^{**}(q)=q\cdot x-w^*(x)\quad \text{ if }q\in\partial w^*(x).
\]
Since $w^*$ is affine on each Dirichlet-Voronoi polytope $\operatorname{Vor}_{\Gamma}(\xi)$, 
\[
\partial w^*(x)=\{\xi\}
\quad \text{for all }\ x\in \operatorname{int}\bigl(\operatorname{Vor}_{\Gamma}(\xi)).
\]
More generally, the subdifferential of $w^*$ at $x\in \R^n$ is
\[
\partial w^*(x)=\operatorname{conv}\{\xi\in \Gamma:\; x\in \operatorname{Vor}_{\Gamma}(\xi)\},
\quad x\in\R^n.
\]  
In particular, for $\xi\in\Gamma$ and  $p\in \operatorname{vert}(\operatorname{Vor}_{\Gamma}(\xi))$, the Voronoi-Delaunay duality implies that
\[
\partial w^*(p)=\operatorname{Del}_{\Gamma}(p),
\]
and thus, for each $q\in \operatorname{Del}_{\Gamma}(p)=\partial w^*(p)$, the Legendre transform gives
\[
w^{**}(q)= q\cdot p-w^*(p)= q\cdot p-\left(\frac12|\xi|^2+\xi\cdot(p-\xi)\right) =\frac{1}{2}|\xi|^2+p\cdot(q-\xi).
\]
Since the cells $\{\operatorname{Del}_{\Gamma}(p)\}_{p\in \operatorname{vert}(\operatorname{Vor}_{\Gamma}(\Gamma))}$ form a decomposition of $\mathbb{R}^n$,
$w^{**}$ agrees with the piecewise-affine function $w$ defined in \eqref{eq:defudirac}  on each cell and hence everywhere. In conclude, 
\[
w\equiv w^{**},
\]
and its value at a given $q$ is independent of the choice of $\xi$ (equivalently, of the choice of a vertex $p$ with $q\in\operatorname{Del}_{\Gamma}(p)$).

Finally, since $q\in \partial w^{*}(x)$ if and only if $x\in \partial w^{**}(q)$, and $\xi \in \partial w^*(x)$ if and only if $x \in \operatorname{Vor}_{\Gamma}(\xi)$, we obtain 
\[
\partial w^{**}(\xi)=\{ x:\; \xi
\in \partial w^{*}(x) \}=\operatorname{Vor}_{\Gamma}(\xi)\quad\text{ for all }\xi \in \Gamma.
\]
This implies 
\[
\det D^2 w^{**}= \sum_{\xi \in \Gamma} \delta_{\xi}.
\]
Therefore, $w^{**}(A^{\frac{1}{2}}x)$ solves \eqref{eq:periodicdelta} in the Alexandrov sense. 
Moreover, it is straightforward to verify that the function  $w^{**}(x)-\frac{1}{2}|x|^2$ is bounded on $\mathbb{R}^n$. Hence, by Theorems \ref{thm:uniqueness within periodic} and \ref{thm:solutions are periodic}, we deduce that  $w^{**}(A^{\frac{1}{2}}x)$
is the unique global convex solution of \eqref{eq:periodicdelta} with quadratic part $\frac12\,x^{\top}Ax$, up to an additive constant. Furthermore, every global convex solution $u$ of~\eqref{eq:periodicdelta} satisfies 
\[
u(x)=w^{**}(A^{\frac{1}{2}}x)+bx+c,
\]
for some $b\in \R^n$ and $c\in \R$, and that $\{\partial u(z)\}_{z\in \mathbb{Z}^n}=\{A^{\frac{1}{2}}\operatorname{Vor}_{\Gamma}(A^{\frac{1}{2}}z)+b\}_{z\in \mathbb{Z}^n }$ forms a tiling of $\mathbb{R}^n$ that is periodic under translations by $A\mathbb{Z}^n$.
\end{proof}

Based on Theorem \ref{thm:periodicdelta-formula}, we obtain an explicit formula for the solution by computing the associated Dirichlet-Voronoi polytope. In general, the computation requires identifying the relevant lattice points in $\Gamma(=A^{\frac{1}{2}}\mathbb{Z}^n)$ that determine the supporting half-spaces (equivalently, the facet-defining inequalities) of the polytope $\operatorname{Vor}_{\Gamma}(0)$. Below, we consider two special cases: when $A$ is a scalar dilation, and when the eigenvalues of $A$ are close to $1$ in two dimensions.

\begin{ex}
In the case $A=\operatorname{diag}\{ \lambda_1,\cdots ,\lambda_n\} \in \mathcal{S}_+^{n\times n}$, the function
\[
u(z+y)=\frac{1}{2}z^{\top}Az+z^{\top}Ay+\frac{1}{2}\sum_{i=1}^{n} \lambda_iy_i\quad \text{for all } z\in \mathbb{Z}^n,\ \ y\in[0,1]^n
\]
is a solution of \eqref{eq:periodicdelta}. Moreover, the sets $\partial u(z) =A \left(\{z\} +[-\frac{1}{2},\frac{1}{2}]^n\right) , z\in \mathbb{Z}^n$  form a tiling of $\R^n$ under the translation by $ A\mathbb{Z}^n $.
\end{ex}

\begin{ex}
In the case of $n=2$ and 
\[ 
A=
\begin{pmatrix}
\cos \theta & -\sin \theta\\
\sin \theta & \cos \theta 
\end{pmatrix}
\begin{pmatrix}
5/4 & 0 \\
0 & 4/5
\end{pmatrix} 
\begin{pmatrix}
\cos \theta & \sin \theta\\
-\sin \theta & \cos \theta 
\end{pmatrix}
= \frac{41}{40} I_2 
+ \frac{9}{40} 
\begin{pmatrix} 
\cos 2\theta & \sin 2\theta \\ 
\sin 2\theta & -\cos 2\theta 
\end{pmatrix},
\] 
let $\{p_i\}$ be the set of the vertices of the polygon 
\[
\begin{split}
&A^{\frac{1}{2}} \operatorname{Vor}_{A^{\frac{1}{2}}\mathbb{Z}^2}(0)\\
&=  \{
x \in \R^2 :\;  x \cdot z\le \frac{1}{2}z^{\top}Az  , \ \forall  z\in \mathbb{Z}^2  \}\\
&=\left\{y\in \R^2 :\; |y_1|\leq \frac{41+9\cos 2\theta}{80} ,\ |y_2|\leq \frac{41-9\cos 2\theta}{80},\ |y_1\pm y_2|\leq \frac{41\pm9\sin 2\theta}{40}  \right\}.
\end{split}
\]
Then, the function
\[
u(z+y)=\frac{1}{2}z^{\top}Az+  z^{\top}A y+  \max_i   p_i \cdot y \quad \text{for all } z \in \mathbb{Z}^2, \ y \in [-\frac{1}{2},\frac{1}{2}]^2
\]
is a solution of \eqref{eq:periodicdelta}.  Moreover, $\partial u(0) =  A^{\frac{1}{2}} \operatorname{Vor}_{A^{\frac{1}{2}}\mathbb{Z}^2}(0)$ and $\{\partial u(z)\}_{z\in \mathbb{Z}^2}=\{\partial u(0)+A z: z\in \mathbb{Z}^2\} $ form a tiling of $\R^2$ under the translation in $A\mathbb{Z}^2 $.
\end{ex}

\bigskip

\noindent \textbf{Data Availability Statement}: Data sharing is not applicable to this article as no
datasets were generated or analysed during the current study.

\bigskip

\noindent \textbf{Declarations} 

\bigskip

\noindent \textbf{Conflict of interest}: All authors declare that they have no conflicts of interest.

\appendix

\section{A strong maximum principle and a calculus lemma}
The following is a result in Caffarelli-Li-Nirenberg \cite{CLN}, see Theorem 3.1 there. We make some improvement on exposition. Let $\mathcal{S}^{n\times n}$ denote the set of symmetric $n\times n$ matrices.

\begin{Theorem}[Theorem 3.1 in \cite{CLN}]\label{thm:strong maximum principle}
Let $F\in C^{0}\big(B_1\times \R\times \R^{n}\times \mathcal{S}^{n\times n}\big)$,
$v\in C^{2}\big(B_1\big)$, and $u$ is lower semicontinuous in $B_1$.
Assume that $u(0)=v(0)$, $u\ge v$ near $0$, the nonlinear operator $F(x,s,p,M)$ is $C^{1}$ in $(s,p,M)$ near
$\big(0, v(0), \nabla v(0), \nabla^{2}v(0)\big)$,
\[
F\big(x,u,\nabla u,\nabla^{2}u\big)\le
F\big(x,v,\nabla v,\nabla^{2}v\big)
\quad\text{in the viscosity sense near }0,
\]
and
\begin{equation}\label{eq:posdef}
\left(\frac{\partial F}{\partial M_{ij}}\right)
\text{ is positive definite near }
\big(0, v(0), \nabla v(0), \nabla^{2}v(0)\big).
\end{equation}
Then $u\equiv v$ near $0$.
\end{Theorem}

\begin{Remark}
It is easy to see from the proof of Theorem \ref{thm:strong maximum principle} that the $C^{1}$
regularity of $F$ can be weakened to Lipschitz regularity, and
\eqref{eq:posdef} can be weakened to that
\[
F(x,s,p,M+N)\ge F(x,s,p,M)+\epsilon_{1} \|N\|
\]
holds for $N\in \mathcal{S}^{n\times n}_+$ satisfying $\|N\|\le\epsilon_1$ and all $\big|(x,s,p,M)-(0,v(0),\nabla v(0),\nabla^2v(0))\big|\le\epsilon_1$ with some positive constant $\epsilon_{1}>0$.
\end{Remark}

We used the following elementary lemma for convex functions. 
\begin{Lemma}\label{lem:uniformly converge implies gradient converge}
Let $\phi$ be a convex function and $P(x) = \frac{1}{2}x^{\top} A x + b x + c$ be a quadratic polynomial with $A\in \mathcal{S}_+^{n\times n}$. If $|\phi - P| \leq \epsilon$ in $B_{2n}(0)$ for some sufficiently small $\epsilon > 0$, then
\[
\|\phi - P\|_{Lip(B_{n}(0))} \leq 2(\operatorname{Tr}A  )^{\frac{1}{2}}\epsilon^{\frac{1}{2}}.
\] 
\end{Lemma}
\begin{proof} 
For any given $y \in B_{n}(0)$ and $p \in \partial \phi(y)$, denote $p - DP(y) = a e$ for some $e \in \mathbb{S}^{n-1}$ and $a \geq  0$. It suffices to estimate $a$. After subtracting the supporting linear function of $P$ at $y$ from both $\phi$ and $P$, we may assume that $P(x) = \frac{1}{2}(x-y)^{\top} A (x-y) $. Then, the convexity of $\phi$ implies that for any $0 < t < n$,
\[
ta = p\cdot te  \leq \phi(y+te)-\phi(y)\leq P(y+te)-P(y)+2\epsilon =
\frac{1}{2}e^{\top} A e t^2+2\epsilon.
\]
Taking $t =2(\epsilon/e^{\top} A e )^{\frac{1}{2}}$, we obtain the desired estimate that $a \leq 2(e^{\top}A e \epsilon)^{\frac{1}{2}} \leq 2(\operatorname{Tr}A  )^{\frac{1}{2}}\epsilon^{\frac{1}{2}} $.
\end{proof}

\section{Quadratic asymptotics from a homogenization estimate}
\label{app:quadratic-asymptotics}

As mentioned earlier, the proof of  Theorem~\ref{thm:homogenization-periodic-measures} follows from that of Theorem 0.3 in \cite{caffarelli2004liouville} with minor modifications. We will sketch the changes below.

\begin{proof}[Sketched Proof of Theorem~\ref{thm:homogenization-periodic-measures}]
Write $\Gamma=\Gamma_{\{\varepsilon_1,\cdots,\varepsilon_n\}}.$
Since $\Omega\subset B_n$ and a translate of $\kappa_\Gamma(\Omega) \overline{D(\Gamma)}$ is
contained in $\Omega$, we have
\[
\sum_{i=1}^n \|\varepsilon_i\| \le C(n)\kappa_\Gamma(\Omega)^{-1}.
\]
Thus it suffices to prove
\begin{equation}\label{C2}
\|\phi-\bar\phi\|_{L^\infty(\Omega)}
\le
C \left(\sum_{i=1}^n \|\varepsilon_i\|\right)^\beta.
\end{equation}

Let $\mathcal Q$ be the collection of sets in  $\Gamma+D(\Gamma)$ 
intersecting $B_n$. If $\Lambda_*$ is sufficiently large, then $\sum_{i=1}^n \|\varepsilon_i\|$ is
sufficiently small, and thus,
\[
\bigcup_{Q\in\mathcal Q}Q\subset B_{n+1}.
\]
Since $\mu(D(\Gamma))=|D(\Gamma)|$, we have
\[
\mu(\Omega)\le \mu(B_n)
\le
\sum_{Q\in\mathcal Q}\mu(Q)
\le C(n).
\]
Therefore, by Alexandrov's maximum principle \eqref{Alexandrov},
\[
|\phi(x)|^n
\le
C(n)\operatorname{diam}(\Omega)^{n-1}
\operatorname{dist}(x,\partial\Omega)\mu(\Omega)
\le
C(n)\operatorname{dist}(x,\partial\Omega).
\]
The same estimate holds for $\bar\phi$. Hence, for all $x\in\Omega$, we have
\begin{equation}\label{C.1}
-C\operatorname{dist}(x,\partial\Omega)^{1/n}
\le
\phi(x)
\le 0\quad \mbox{and}\quad
-C\operatorname{dist}(x,\partial\Omega)^{1/n}
\le
\bar\phi(x)
\le 0.
\end{equation}

We first assume that $\mu$ has a smooth positive density $\rho$, that is, 
\[
\mu=\rho\,\ud x.
\]
Then, once we have \eqref{C.1}, the rest proof of
\eqref{C2} is identical to that of Theorem 0.3 in \cite{caffarelli2004liouville} and we omit here. There constants $C$ and $\beta$ are positive and depending only on $n$.

It remains to pass to a general nonnegative periodic Borel measure. Let
$\mu_j=\rho_j\,dx$ be the smooth positive periodic approximations constructed
in Remark~\ref{rem:approximated construction}. After multiplying $\rho_j$ by
constants tending to $1$, we may assume
\[
\mu_j(D(\Gamma))=|D(\Gamma)|,
\qquad
\mu_j\rightharpoonup \mu
\]
weakly as measures. Let $\phi_j$ solve
\[
\det D^2\phi_j=\mu_j
\quad \text{in }\Omega,
\qquad
\phi_j=0
\quad \text{on }\partial\Omega.
\]
The estimate \eqref{C.1} is uniform in $j$. By the stability of Alexandrov solutions
under weak convergence of Monge-Amp\`ere measures, $\phi_j\to\phi$ uniformly
on $\overline\Omega$. The final conclusion follows by sending $j\to\infty$.
\end{proof}

The following is an extension of Proposition 2.1 in Caffarelli-Li \cite{caffarelli2004liouville} to general periodic measures.

\begin{Proposition}\label{prop:quadratic-asymptotics}
Let $u$ be a convex function in $\R^n$ whose Monge-Amp\`ere measure $\mu$ is a nonzero $\Gamma_{\{\varepsilon_1,\cdots,\varepsilon_n\}}$-periodic Borel measure satisfying $|D(\Gamma)|  = \mu(D(\Gamma))$. 
Then there exist a unique matrix $A\in \mathcal{S}_+^{n\times n}$ with $\det A =1$, and positive constants $C$ and $\sigma$ (depending on $u$) such that  and
\[
\left|
u(x)-\frac12 x^{\top}Ax
\right|
\le C|x|^{2-\sigma}
\quad \text{for all } |x| \geq 1.
\] 
\end{Proposition}

The proof follows from a scaling argument, the homogenization estimate in Theorem~\ref{thm:homogenization-periodic-measures}, and the general iteration scheme of Caffarelli in \cite{caffarelli1990interiorw2p}. 
See also the proofs of 
Propositions 3.3, 3.4 and (3.10) in \cite{caffarellii2003extension}. Similar proof has also been given in the proof of Proposition \ref{prop:basic}.

\begin{proof}[Proof of Proposition~\ref{prop:quadratic-asymptotics}]
After subtracting an affine function, assume
\[
u(0)=0,
\qquad
0\in\partial u(0).
\]
Thus $u\ge0$ in $\R^n$. For $h>1$, set
\[
S_h:=\{x\in\R^n:u(x)<h\}.
\]
By Lemma~\ref{lem:no lay}, each $S_h$ is bounded.

We first recall the large-section estimates needed below. For each large $h$,
choose a linear transformation $\mathcal T_h$ such that
\[
B_1\subset \Omega_h:=\mathcal T_h^{-1}S_h\subset B_n.
\]
Consider the boundary-zero normalized function
\[
u_h:=\frac{u(\mathcal T_hy)-h}{|\det D\mathcal T_h|^{2/n}}
\quad \text{in }\Omega_h.
\]
Its Monge--Amp\`ere measure is periodic with respect to the rescaled lattice
$(D\mathcal T_h)^{-1}\Gamma$, and
\[
\kappa_{(D\mathcal T_h)^{-1}\Gamma}(\Omega_h)=\kappa_\Gamma(S_h).
\]
Since the sections $S_h$ exhaust $\R^n$, we have
$\kappa_\Gamma(S_h)\to\infty$ as $h\to\infty$. Thus, for $h$ sufficiently large,
Proposition~\ref{prop:basic} applies to the above normalized function $u_h$.

Since $u\ge0$ and $u(0)=0$, we have $\|u_h\|_{L^\infty(\Omega_h)}=\frac{h}{|\det D\mathcal T_h|^{2/n}}$. The uniform height bounds in part (i) of
Proposition~\ref{prop:basic} (with $\phi=u_h$ and $\Omega=\Omega_h$) gives
\[
\|u_h\|_{L^\infty(\Omega_h)}\ge c,
\] 
and Alexandrov's estimate
\eqref{Alexandrov} gives
\[
\|u_h\|_{L^\infty(\Omega_h)}
\le C\quad\mbox{and also}\quad \mbox{dist}(\mathcal T_h^{-1}(0),\partial\Omega_h)\ge c,
\]
where $c$ and $C$ are positive constants depending only on $n$. 
Therefore,
\begin{equation}\label{eq:voltram}
c\le\frac{h}{|\det D\mathcal T_h|^{2/n}}
\le C.
\end{equation}
As $|\Omega_h|$ is bounded above and below by dimensional constants, we also get
\[
c h^{n/2}\le |S_h|\le C h^{n/2}.
\]

Fix $\Lambda_u>0$ sufficiently large so that
$\kappa_\Gamma(S_{\Lambda_u})\ge \Lambda_0^2$. For $h\gg\Lambda_u$, the set
$\mathcal T_h^{-1}S_{\Lambda_u}$ is a section of $u_h$ at $\mathcal T_h^{-1}(0)$ of height $\Lambda_u/|\det D\mathcal T_h|^{2/n}$, that is
comparable to $\Lambda_u/h$ (by \eqref{eq:voltram}). Applying part (ii) of 
Proposition~\ref{prop:basic} to the section $\mathcal T_h^{-1}S_{\Lambda_u}$ with $\phi=u_h$ and $\Omega=\Omega_h$, and then pulling back by $\mathcal T_h$, we obtain
\[
c\left(\frac{h}{\Lambda_u}\right)^{1/2}S_{\Lambda_u}
\subset
S_h
\subset
C\left(\frac{h}{\Lambda_u}\right)^{1/2}S_{\Lambda_u},
\]
where the dilations are taken with respect to the origin. In particular,
\[
\kappa_\Gamma(S_h)
\ge
c\left(\frac{h}{\Lambda_u}\right)^{1/2}
\kappa_\Gamma(S_{\Lambda_u})
\ge
c h^{1/2}
\]
for all sufficiently large $h$.

The rest proof is almost identical to that of Proposition 2.1 in \cite{caffarelli2004liouville} using Caffarelli's large-section iteration. See also the proofs of 
Propositions 3.3, 3.4 and (3.10) in \cite{caffarellii2003extension}, as well as our proof of Proposition \ref{prop:basic}. We just sketch the details for completeness.

Choose $r>1$ large and set
\[
h_k:=r^k.
\]
For each large $k$, define the quadratic excess
\[
E_k
:=
\inf_{P\in\mathcal P}
\frac1{h_k}\|u-P\|_{L^\infty(S_{h_k})},
\]
where 
\[
\mathcal P
:=
\left\{
P(x)=\frac12 x^\top A_Px+\ell_P(x)
:\ A_P\in\mathcal S_+^{n\times n},\ \det A_P=1,\ 
\ell_P \ \text{is affine}
\right\}.
\]

By the same proof of Proposition~\ref{prop:basic}, one can show that there exist $\theta\in(0,1)$ and $C>0$ such that, for all large
$k$,
\begin{equation}\label{C.14}
E_{k-1} \le \theta E_k+C \kappa_\Gamma(S_{h_k})^{-\beta}\le  \theta E_k+C h_k^{-\frac{\beta}{2}}.
\end{equation}

Choose $\sigma\in(0,\beta/2)$ sufficiently small so that $\theta r^{-\sigma}<1$.
Iterating \eqref{C.14} gives
\[
E_k\le C h_k^{-\sigma}.
\]
Thus, for each large $k$, there exists a quadratic polynomial
\[
P_k(x)=\frac12 x^\top A_kx+\ell_k(x),
\qquad
A_k\in\mathcal S_+^{n\times n},
\qquad
\det A_k=1,
\]
such that
\[
\|u-P_k\|_{L^\infty(S_{h_k})}
\le
C h_k^{1-\sigma}.
\]
As in the  proof of Proposition~\ref{prop:basic}, one can remove the affine part of $P_k$ and obtain
\begin{equation}\label{C.16}
\left|
u(x)-\frac12 x^\top A_kx
\right|
\le
C h_k^{1-\sigma/2}
\qquad
\text{for }x\in S_{h_k}.
\end{equation}
By comparing the quadratic parts at two consecutive scales
$S_{h_k}\subset S_{h_{k+1}}$, we obtain
\begin{equation}\label{C.17}
\left\|
A_{k}^{-1/2}A_{k+1}A_{k}^{-1/2}-I
\right\|
\le
C h_k^{-\sigma/2}.
\end{equation}
Equivalently,
\[
(1-C h_k^{-\sigma/2})A_{k}
\le
A_{k+1}
\le
(1+C h_k^{-\sigma/2})A_{k}
\]
for all sufficiently large $k$. Since $\sum_k h_k^{-\sigma/2}<\infty$, iteration gives uniform ellipticity of
$\{A_k\}$. In particular, \eqref{C.17} implies
\[
\|A_{k+1}-A_k\|
\le
C h_k^{-\sigma/2}.
\]
Thus $\{A_k\}$ is a Cauchy sequence. Let $A:=\lim_{k\to\infty}A_k$. Then
$A\in\mathcal S_+^{n\times n}$ and $\det A=1$. Moreover,
\begin{equation}\label{C.18}
\|A-A_k\|
\le
C h_k^{-\sigma/2}.
\end{equation}

Let $x\in S_{h_k}\setminus S_{h_{k-1}}$. By the ellipsoid comparison above and the uniform ellipticity of $A_k$,
we have $ch_{k-1}\le |x|^2\le C h_k$, that is, $cr^{-1}h_{k}\le |x|^2\le C h_k$. Therefore, using \eqref{C.18},
\[
\left|x^\top(A-A_k)x\right|
\le C h_k^{1-\sigma/2}.
\]
This, together with \eqref{C.16}, implies
\[
\left|
u(x)-\frac12 x^\top Ax
\right|
\le
C h_k^{1-\sigma/2}
\qquad
\text{for }x\in S_{h_k}\setminus S_{h_{k-1}}.
\]
Hence,
\[
\left|
u(x)-\frac12 x^\top Ax
\right|
\le
C(1+|x|)^{2-\sigma}.
\]

Finally, the matrix $A$ is unique. Indeed, if $A_1,A_2$ both satisfy the
conclusion, then
\[
|x^\top(A_1-A_2)x|\le C|x|^{2-\sigma}
\quad \text{for }|x|\ge1.
\]
Since $0<\sigma<1$, we obtain $A_1=A_2$ by sending $x\to\infty$. This completes the proof.
\end{proof}

\small


\smallskip

\noindent T. Jin

\noindent Department of Mathematics, The Hong Kong University of Science and Technology\\
Clear Water Bay, Kowloon, Hong Kong\\
Email: \textsf{tianlingjin@ust.hk}

\bigskip

\noindent Y.Y. Li

\noindent Department of Mathematics, Rutgers University\\ 
110 Frelinghuysen Road, Piscataway, NJ 08854-8019, USA\\
Email: \textsf{yyli@math.rutgers.edu}

\bigskip

\noindent H.V. Tran

\noindent Department of Mathematics, University of Wisconsin at Madison,\\
Van Vleck Hall, 480 Lincoln Drive, Madison, WI 53706, USA\\
Email: \textsf{hung@math.wisc.edu}

\bigskip

\noindent X. Tu

\noindent Department of Mathematics, The Hong Kong University of Science and Technology\\
Clear Water Bay, Kowloon, Hong Kong\\[1mm]
Email:  \textsf{maxstu@ust.hk}

\end{document}